\documentclass[12pt]{article}
\usepackage{amsthm}

\usepackage[ansinew]{inputenc}
\usepackage{graphicx}
\usepackage{color}
\usepackage[colorlinks]{hyperref}

\usepackage{enumerate,latexsym}
\usepackage{latexsym}
\usepackage{amsmath,amssymb}
\usepackage{graphicx}
\usepackage{times}
\usepackage{float}
\usepackage[colorlinks]{hyperref}
\usepackage{hyperref}
\hypersetup{%
	colorlinks = true,
	linkcolor  = black
}
\newfont{\bb}{msbm10 at 11pt}
\newfont{\bbsmall}{msbm8 at 8pt}

\def\rth{\mathbb{R}^3}
\def\R{\mathbb{R}}
\def\B{\mathbb{B}}
\def\N{\mathbb{N}}

\def\S{\Sigma}

\def\Hip{\mathbb{H}}

\def\D{\mathbb{D}}

\def\LM{\mathbb{L}}

\newcommand{\la}{\looparrowright}

\newcommand{\ben}{\begin{enumerate}}
	\newcommand{\bit}{\begin{itemize}}
		\newcommand{\een}{\end{enumerate}}
	\newcommand{\eit}{\end{itemize}}
\newcommand{\wh}{\widehat}
\newcommand{\ds}{\displaystyle}
\newcommand{\Int}{\mbox{\rm Int}}
\newcommand{\Ind}{\mbox{\rm Index}}
\newcommand{\Inj}{\mbox{\rm Inj}}

\newcommand{\wt}{\widetilde}

\newcommand{\ed}{\end{document}}
\newcommand{\ov}{\overline}

\def\a{{\alpha}}

\def\g{{\gamma}}
\def\G{{\Gamma}}
\def\l{{\lambda}}
\def\L{\Lambda}
\def\de{{\delta}}
\def\be{{\beta}}
\def\ve{{\varepsilon}}

\def\cP{\mathcal{P}}

\def\cC{\mathcal{C}}

\let\8=\infty \let\0=\emptyset

\def\cte.{\mathop{\rm cte.}\nolimits}

\def\det{\mathop{\rm det}\nolimits}

\def\cosh{\mathop{\rm cosh }\nolimits}
\def\tanh{\mathop{\rm tanh }\nolimits}

\def\arccosh{\mathop{\rm arccosh }\nolimits}

\def\N{\mathbb{N}}

\def\B{\mathbb{B}}

\def\R{\mathbb{R}}

\def\D{\mathbb{D}}

\def\H{\mathbb{H}}

\newtheorem{theorem}{Theorem}[section]
\newtheorem{lemma}[theorem]{Lemma}
\newtheorem{proposition}[theorem]{Proposition}

\newtheorem{remark}[theorem]{Remark}

\newtheorem{definition}[theorem]{Definition}

\newtheorem{claim}[theorem]{Claim}

\numberwithin{equation}{section}

\usepackage{fullpage}

\usepackage{pdfsync}
\definecolor{pp}{rgb}{.5,0,.7}

\begin{document}
\begin{title}
{Geometry of CMC  surfaces of finite index}
\end{title}

\begin{author}
{William H. Meeks III$^*$
 \and Joaqu\'\i n P\' erez\thanks{Research of both
authors was partially supported by MINECO/MICINN/FEDER grant no.
PID2020-117868GB-I00, regional grants P18-FR-4049 and A-FQM-139-UGR18,
and by the “Maria de Maeztu” Excellence Unit IMAG,
reference CEX2020-001105-M, funded by MCINN/AEI/10.13039/501100011033/ CEX2020-001105-M.}}
\end{author}
\maketitle
\vspace{-.6cm}

\begin{abstract}
Given $r_0>0$,  $I\in \N\cup \{0\}$ and $K_0,H_0\geq0$,
let $X$ be a complete Riemannian $3$-manifold with
injectivity radius  $\Inj(X)\geq r_0$ and with the  supremum
of absolute sectional curvature  at most
$K_0$, and let $ M \looparrowright X$ be a complete
immersed surface of constant mean curvature
$H\in [0,H_0]$ and with index at most $I$.
We will obtain geometric estimates for such an $ M \looparrowright X$
as a consequence of the
Hierarchy Structure Theorem in~\cite{mpe18}.
The Hierarchy Structure Theorem (see Theorem~\ref{mainStructure} below)
will be applied to understand global
properties of $ M \looparrowright X$, especially results related to
the area and diameter of $M$. By item~\ref{itF} of
Theorem~\ref{mainStructure}, the area of such a non-compact  $ M \looparrowright X$
is infinite.  We will improve this area result by proving the following
when  $M$ is connected; here  $g(M)$ denotes the genus of the orientable cover of $M$:
\ben
\item There exists $C_1=C_1(I,r_0,K_0,H_0)>0$,
such that $\mbox{\rm Area}(M)\geq C_1(g(M)+1)$.
\item There exist   $C>0$, $G(I)\in \N$  independent
of $r_0,K_0,H_0$ and also $C$ independent of $I$
such that if $g(M)\geq G(I)$, then
$\mbox{\rm Area}(M)\geq \frac{C}{(\max\{1,\frac{1}{r_0},\sqrt{K_0}, H_0\})^2}(g(M)+1)$.
\item If the scalar curvature $\rho$ of
$X$ satisfies  $3H^2+\frac{1}{2}\rho\geq c$ in $X$ for some $c>0$, then
there exist $A,D>0$ depending on $c,I,r_0,K_0,H_0$ such that
$\mbox{\rm Area}(M)\leq A$ and  $ \mbox{\rm Diameter}(M)\leq D$.
Hence, $M$ is compact and, by item~1,   $ g(M)\leq A/C -1.$
\een
\vspace{.1cm}


\end{abstract}
\maketitle

\section{Introduction}  \label{sec:introduction}
Throughout the paper,  $X$ denotes a complete Riemannian $3$-manifold with positive
injectivity radius $\Inj(X)$ and bounded absolute sectional curvature.
Let $M$ be a complete immersed surface in $X$ of constant mean curvature
$H\geq 0$, which we call
an {\it $H$-surface} in $X$. The Jacobi operator of $M$ is the
Schr\"{o}dinger operator
\[
L=\Delta +|A_M|^2+ \mbox{Ric}(N),
\]
where
$\Delta $ is the Laplace-Beltrami operator on $M$,
$|A_M|$ is the
norm of its second fundamental form and $\mbox{Ric}(N)$ denotes
the Ricci curvature of $X$ in the direction of the unit normal vector $N$ to $M$;
the index of $M$ is the index of $L$,
\[
\mbox{Index}(M)=\lim _{r\to \infty }\mbox{Index}(B_M(p,r)),
\]
where $B_M(p,r)$ is the intrinsic metric ball in $M$ of
radius $r>0$ centered at a point $p\in M$, and
$\mbox{Index}(B_M(p,r))$ is the number of negative eigenvalues of $L$
on $B_M(p,r)$ with Dirichlet boundary conditions. Here, we have assumed that the immersion
is two-sided (this holds in particular if $H>0$).  In the case, $H=0$ and the immersion is
one-sided, then the index is defined in a similar manner using compactly supported
variations in the normal bundle; see Definition~\ref{DefIndexNO} for details.

The primary goal of this paper is to apply the Hierarchy Structure
Theorem~\ref{mainStructure} (proven in~\cite{mpe18})
to
understand certain global properties of closed constant
mean curvature surfaces in Riemannian 3-manifolds.
Theorem~\ref{mainStructure}
describes the geometric structure of complete immersed
$H$-surfaces $F\colon M\la X$ (also called {\it $H$-immersions})
which have a fixed  bound $I\in \N\cup \{0\}$ on their
index and a fixed upper bound $H_0$ for their constant mean curvature $H
\geq 0$, in certain small
intrinsic neighborhoods of points with sufficiently
large norm $|A_M|$ of their second fundamental forms.

Our main applications of Theorem~\ref{mainStructure} appear
in Theorem~\ref{main2}  and Theorem~\ref{main22};
these two theorems provide lower bounds for the areas and intrinsic
diameters of  immersed closed $H$-surfaces $M$ in $X$ of finite
index in terms of their genera, when  the indices and the constant
mean curvatures of the surfaces are bounded from above by fixed
constants. Theorem~\ref{main22} also provides upper bounds for the
area of balls $B_M(x,r)$ in $M$ for every $x\in M$ and $r>0$,
independently on whether or not $M$ is compact but depending on
upper bounds for $H$ and the index of $M$.

In the case that $M$ is non-orientable,
the genus $g(M)$ of $M$ is the genus of its oriented  cover.

\begin{theorem}[Area and diameter estimates]
\label{main2}
For $r_0>0$, $K_0, H_0\geq 0$, consider all
complete Riemannian  3-manifolds $X$ with injectivity radius
$\Inj(X)\geq r_0$ and absolute sectional curvature bounded from above by
$K_0$, and let $\lambda=\max\{1,\frac{1}{r_0},\sqrt{K_0}, H_0\}$.
Let $M$ be a complete immersed $H$-surface in $X$ with empty boundary,
$H\in [0,H_0]$, index at most~$I\in \N\cup \{ 0\}$ and genus $g(M)$,
which in the language of Theorem~\ref{mainStructure} implies
$M\in \Lambda =\L(I, H_0, r_0, 1, K_0)$ with additional chosen
constant $\tau=\pi/10$. Then:
\ben
\setcounter{enumi}{-1}
\item
\label{It0}
The area of $M$ is greater than $C_A/\l^2$, where
\[
C_A:=\pi\, \left( \frac{\pi}{4}\right)^2 e^{-\frac{\pi}{2}-1+\frac{\pi}{4}} \approx 0.325043,
\]
and if $M$ is compact, the extrinsic diameter of each component of
$M$ is greater than $\frac{\pi}{4\l}$.

\item
\label{It1a} {\rm (Item~1 in the abstract).}
There exists $C_1(I)>0$ (independent of  $M,r_0,K_0,H_0$) such that:
\begin{equation}
\label{1.1}
\mbox{\rm Area}(M)\geq \frac{C_1(I)}{\lambda^2}(g(M)+1).
\end{equation}

\item  \label{It1b}
{\rm (Item~2 in the abstract).}
Let   $C_s\geq 2\pi$ be the
universal curvature estimate for stable $H$-surfaces described in
Theorem~\ref{stableestim1s} below and let
$C=\pi/(3+4C_s+4C_s^2)$. There exists a $G(I)\in \N$, so that whenever $g(M)\geq G(I)$, then:
\begin{equation}
	\label{1.1b}
	\mbox{\rm Area}(M)\geq \frac{C}{\l^2}(g(M)+1).
\end{equation}

%
	
\item \label{It3b}
{\rm (Item~3 in the abstract).}
Suppose that  the scalar curvature $\rho$ of $X$ satisfies $3H^2+\frac{1}{2}\rho
\geq c$ for some $c>0$. Then, if $M$
is connected, then $M$ is compact, and furthermore, there exists $A_2(I,c)>0$  such that:
\begin{equation}
\label{1.7}
\mbox{\rm Area}(M)\leq \frac{A_2(I,c)}{\l^2}  \qquad
\mbox{\rm Diameter}(M)\leq \frac{ 4\pi(I+1)}{\l\sqrt{3c}},
\qquad
g(M)\leq \frac{A_2(I,c)}{C_1(I)}-1.
\end{equation}
\een
\end{theorem}

In the proof of Theorem~\ref{main2} the estimates for
the constants $C_1(I)$, 
 $G(I)$,
and $A_2(I,c)$ 
will be given in terms of the related constants  $A_1(I)$, $\de(I)$
 given in the Hierarchy Structure Theorem (for the value 
$\tau =\frac{\pi}{10}$)
for the space $\L(I,1,1,1,1)$ described in Definition~\ref{def:L}.


There are a number of recent results in the literature
related to area estimates for connected, closed,
embedded minimal and CMC surfaces of finite index
in a closed 3-dimensional  Riemannian manifold, some of which include
results described in Theorem~\ref{main2} under more restrictive geometric hypotheses
on the surfaces and/or the ambient space.
Some of these recent results, obtained independently,
can be found in  the papers
\cite{AiHo1,abcs1,bst1,bush1,ChKeMa1,max1,mt12,sat1}.
We refer the interested reader to  \cite{bst1}
for further references in this active research area and for the
general historical background that motivates this  subject material.
\par
\vspace{.2cm}
\noindent{\bf Acknowledgments:}  The authors would like to thank Harold Rosenberg
for his thoughts and discussions of our initial attempts at understanding
the existence of the linear area estimates
given in item~\ref{It1a} of Theorem~\ref{main2}.


\section{The Hierarchy Structure Theorem}
\label{Sec:structure}

In the sequel, we will denote by $B_X(x,r)$ (resp. $\ov{B}_X(x,r)$)
the open (resp. closed) metric ball centered at a point
$x\in X$ of radius $r>0$. For a Riemannian surface $M$ with smooth compact boundary $\partial M$,
\[
\kappa(M)=\int_{\partial M}\kappa_g,
\]
will  stand for the total geodesic curvature of $\partial M$, where $\kappa _g$ denotes the
pointwise geodesic curvature of $\partial M$  with
respect to the inward pointing unit conormal vector of $M$ along $\partial M$.

\begin{definition}
\label{def:L}
{\rm
For every $I\in \N\cup \{ 0\}$,   $\ve_0>0$, and $H_0,A_0,K_0\geq 0$, we
denote by
\[
\L=\L(I, H_0,\ve_0,A_0,K_0)
\]
the space of all $H$-immersions $F\colon M\la X$
satisfying the following conditions:
\begin{enumerate}[({A}1)]
\item $X$ is a complete Riemannian 3-manifold with injectivity radius
$\Inj(X)\geq \ve_0$
and absolute sectional curvature bounded from above by $K_0$.
\item $M$ is a complete surface with smooth boundary (possibly empty)
and when $\partial M\neq \varnothing$,
there is at least one point in $M$ of distance $\ve_0$ from $\partial M$.
\item $H\in [0, H_0]$ and $F$ has index at most $I$.
\item  If $\partial M\neq \varnothing$, then for any $\ve \in
(0,\infty ]$ we let
$U(\partial M,\ve)=\{x\in M\mid d_M(x,\partial M)<\ve \}$
be the open intrinsic $\ve $-neighborhood of $\partial M$. Then,
$|A_M|$ is bounded from above by $A_0$ in $U(\partial M,\ve_0)$.
\end{enumerate}
}
\end{definition}

Suppose that $(F\colon M\la X)\in \L$ and $\partial M\neq \varnothing$.
For any positive $\ve_1 \leq \ve_2\in [0,\infty ]$, let
\[
U(\partial M,\ve_1,\ve_2 )=U(\partial M,\ve_2)\setminus
\overline{U(\partial M,\ve_1)},\quad
\ov{U}(\partial M,\ve_1,\ve_2 )=\ov{U(\partial M,\ve_2)}\setminus
U(\partial M,\ve_1).
\]
When $\partial M=\varnothing$, we  define $U(\partial M,\ve_1,\infty )=
\ov{U}(\partial M,\ve_1,\infty )$  as $M$.

In the next result we will make use of harmonic
coordinates $\varphi_x\colon U\to B_X(x,r)$ defined on an open subset $U$ of $\R^3$
containing the origin, taking values
in a geodesic ball $B_X(x,r)$ centered at a point $x\in X$ of radius
$r\in (0,\Inj_X(x))$ (here, $\Inj_X(x)$ stands for the injectivity radius
of $X$ at $x$) and with a
$C^{1,\a}$ control of the ambient metric on $X$, see
Definition~\ref{defharm} for details.

\begin{theorem}[Structure Theorem for finite index $H$-surfaces~\cite{mpe18}]
\label{mainStructure}
Given $\ve_0>0$, $K_0, H_0, A_0\geq 0$, $I\in \N\cup \{0\}$ and
$\tau \in (0,\pi /10]$, there exist  $A_1\in [A_0,\infty)$,
$\de_1,\de\in (0,\ve_0/2]$ with $\de_1\leq \de/2$,
such that the following hold:
	
For any $(F\colon M\la  X)\in \L=\L(I, H_0,\ve_0,A_0,K_0)$,
there exists 
a (possibly empty) finite collection $\cP_F=\{p_1,\ldots,p_k\}\subset
U(\partial M,\ve_0 ,\infty )$ of points, $k\leq I$, and
numbers $r_F(1),\ldots ,{r_F}(k)\in [\de_1,\frac{\de}{2}]$
with $r_F(1)>4r_F(2)>\ldots >4^{k-1}r_F(k)$, satisfying
the following:
\begin{enumerate}
\item
\label{it1}
\underline{Portions with concentrated curvature:}
Given $i=1,\ldots ,k$, let $\Delta_i$ be the component
of \newline
$F^{-1}(\ov{B}_X(F(p_i),r_F(i)))$ containing $p_i$.
Then:
\begin{enumerate}[a.]
\item $\Delta _i\subset \ov{B}_M(p_i,\frac{5}{4}r_F(i))$ (in particular,
$\Delta _i$ is compact).
\item $\Delta _i$ has smooth boundary
and $F(\partial \Delta_i)\subset \partial \ov{B}_X(F(p_i),r_F(i))$.
\item  $B_M(p_i,\frac{7}{5}r_F(i))\cap
B_M(p_j,\frac{7}{5}r_F(j))=\varnothing $ for $i\neq j$.
In particular, the intrinsic distance between $\Delta _i,\Delta _j$ is
greater than $\frac{3}{10}\de_1$ for every $i\neq j$.
\item $|A_M|(p_i)=\max_{\Delta_i}|A_M|= \max \{ |A_M|(p)\ :\ p\in M
\setminus \cup _{j=1}^{i-1}B_M(p_j,\frac54 r_F(j))\} \geq  A_1$,
see Figure~\ref{fig1}.
\begin{figure}
\begin{center}
\includegraphics[width=11cm]{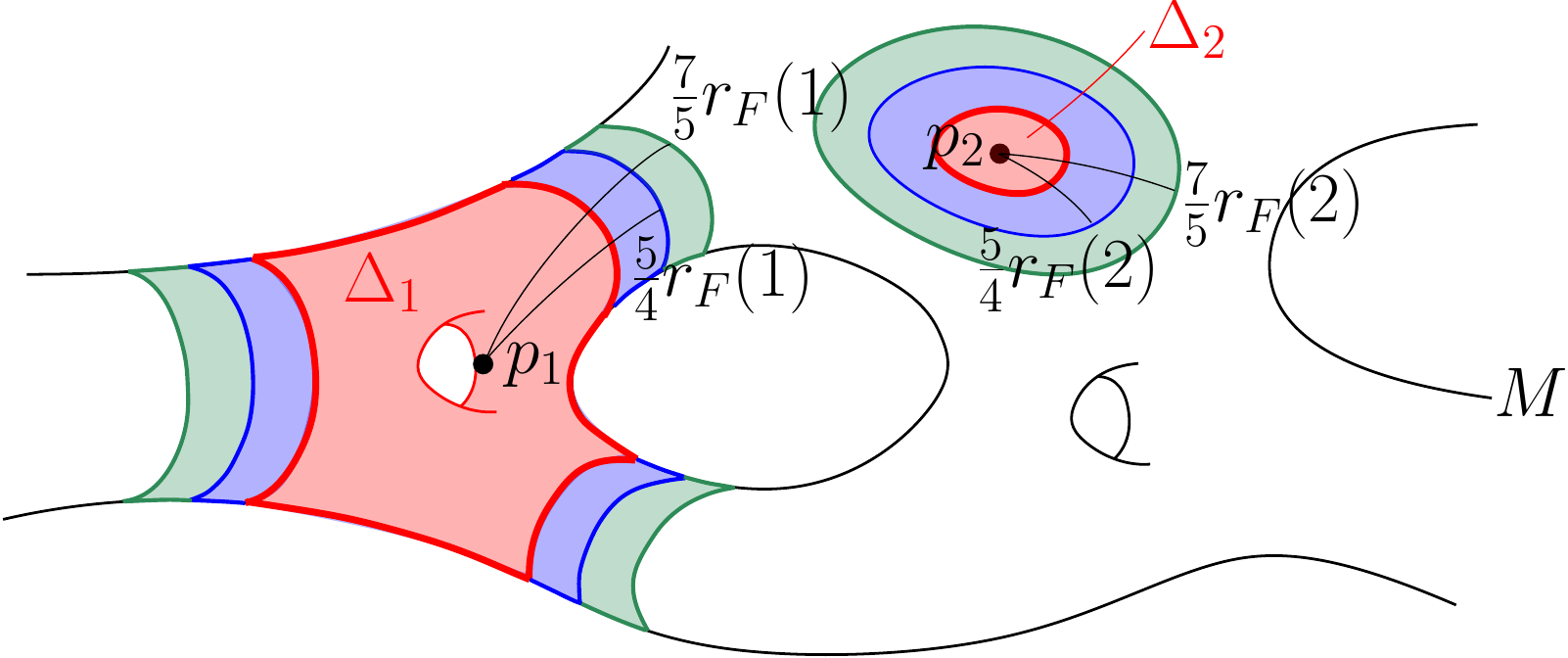}
\caption{The second fundamental form concentrates
inside the intrinsic compact regions $\Delta_i$ (in red),
each of which is mapped through the immersion $F$
to a surface inside the extrinsic ball in $X$ centered at
$F(p_i)$ of radius $r_F(i)>0$, with $F(\partial \Delta_i)\subset \partial
\ov{B}_X(F(p_i),r_F(i))$.
Although the boundary $\partial \Delta_i$ might not be at constant
intrinsic distance
from the `center' $p_i$, $\Delta_i$ lies entirely inside the
intrinsic ball centered at $p_i$ of radius $\frac{5}{4}r_F(i)$.
The intrinsic open balls $B_M(p_i,\frac{7}{5}r_F(i))$ are pairwise
disjoint.}
\label{fig1}
\end{center}
\end{figure}
			
\item The index  $\Ind(\Delta_i)$ of $\Delta_i$ is positive.
\end{enumerate}
		
\item
\label{it3}
\underline{Transition annuli:}
For $i=1,\ldots ,k$ fixed, let $e(i)\in \N$ be the number of boundary
components of $\Delta _i$.
Then, there exist planar disks $\D_{1},\ldots ,\D_{e(i)}\subset
T_{F(p_i)}X$ of radius $2r_F(i)$ centered at the origin in $T_{F(p_i)}X$,
such that if we denote by
\[
P_{i,h}=\varphi _{F(p_i)}(\D _h),\quad h\in \{ 1,\ldots ,e(i)\} ,
\]
(here $\varphi_{F(p_i)}$ denotes  a harmonic chart
centered at $F(p_i)$, see Definition~\ref{defharm}),
then
\[
F(\Delta_i)\cap[\ov{B}_X(F(p_i),r_F(i))\setminus  B_X(F(p_i),r_F(i)/2)]
\]
consists of $e(i)$ annular multi-graphs\footnote{See
Definition~\ref{DefMulti} for this notion of multi-graph.}
$G_{i,1},\ldots,G_{i,e(i)}$ over their projections
to $P_{i,1},\ldots ,P_{i,e(i)}$, with multiplicities $m_{i,1},\ldots
m_{i,e(i)}\in \N$ respectively,
and whose related graphing functions $u$ satisfy
\begin{equation}
\frac{|u(x)|}{|x|}+|\nabla u|(x)\leq \tau ,
\label{estimu}
\end{equation}
where we have taken coordinates $x$ in each of the $P_{i,h}$ and denoted
by $|x|$ the extrinsic distance to $F(p_i)$ in the ambient metric of $X$,
see Figure~\ref{fig2}.
\begin{figure}
\begin{center}
\includegraphics[width=10cm]{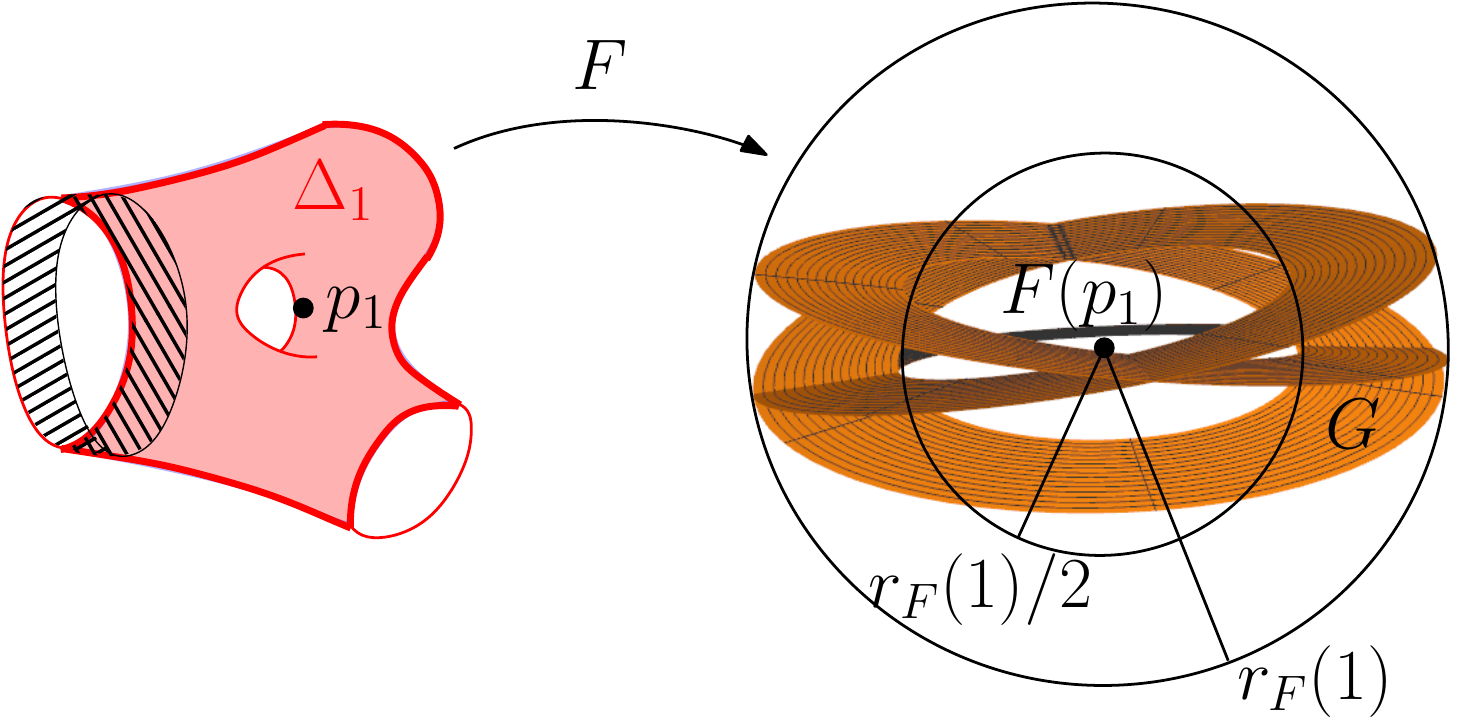}
\caption{The transition annuli: On the right, one has the extrinsic
representation in $X$ of one of the annular multi-graphs $G$ in
$F(\Delta_1)\cap[\ov{B}_X(F(p_1),r_F(1))\setminus  B_X(F(p_1),r_F(1)/2)]$;
in this case, the multiplicity of the
multi-graph is 3. On the left, one has the intrinsic representation of
the same annulus (shadowed); there is one such annular multi-graph for
each boundary component of $\Delta _i$.}
\label{fig2}
\end{center}
\end{figure}		
\item \label{it7}
\underline{Region with uniformly bounded curvature:}
$|A_{M}|< A_1$  on $\wt{M}:=M\setminus \bigcup_{i=i}^k \Int(\Delta_i)$.
\end{enumerate}
	
\noindent
Moreover, the following additional properties hold:
\begin{enumerate}[A.]
\item  \label{it2}
$\sum_{i=1}^k I(\Delta_i)\leq I$, where $I(\Delta_i)=\Ind(\Delta_i)$.
	
\item \label{it4} \underline{Geometric and topological estimates:} Given
$i=1,\ldots ,k$, let $m(i):=\sum_{h=1}^{e(i)}m_{i,h}$ be the total
spinning of the boundary of $\Delta_i$, let $g(\Delta_i)$ denote the
genus of $\Delta_i$
(in the case $\Delta_i$ is non-orientable, $g(\Delta_i)$ denotes the
genus of its oriented cover\footnote{If $\Sigma$ is
a compact non-orientable
surface and	$\wh{\S}\stackrel{2:1}{\to}\S$ denotes the
oriented cover of $\S$, then the genus of $\wh{\S}$ plus 1 equals the
number of cross-caps in $\S$.}). Then, $m(i)\geq 2$ and the following
upper estimates hold:
\begin{enumerate}
	\item If $I(\Delta_i)=1$, then $\Delta_i$ is
		orientable, $g(\Delta_i)=0$, and $(e(i),m(i))
		\in \{ (2,2),(1,3)\}$.
	
	\item If $I(\Delta_i)\geq 2$ and
	$\Delta_i$ is orientable, then
	$m(i)\leq 3I(\Delta_i)-1$,
        $e(i)\leq 3I(\Delta_i)-2$,
	and $g(\Delta_i)\leq 3I(\Delta_i)-4$.
	
	\item If $\Delta_i$ is non-orientable, then $I(\Delta_i)\geq 2$,
	$m(i)\leq 3I(\Delta_i)-1$,
	$e(i)\leq 3I(\Delta_i)-2$ and
	$g(\Delta_i)\leq 6I(\Delta_i)-8$.
	
\item $\chi(\Delta_i)\geq -6I(\Delta_i) +2m(i)+e(i)$
and thus, $\chi(\cup_{i=1}^k \Delta_i) \geq -6I +2S+e$, where
\[
e=\sum_{i=1}^ke(i),\qquad S=\sum_{i=1}^km(i).
\]

\item $|\kappa(\Delta_i)-2\pi m(i)|\leq \frac{\tau}{m(i)}$,
and so, the total geodesic curvature $\kappa(\widetilde{M})$
of $\widetilde{M}$ along $\partial \wt{M}\setminus \partial M$ satisfies
$\left| \kappa(\wt{M})+2\pi S\right| \leq
\frac{\tau}{2}k$, 
and so,
\begin{equation}
	2\pi S-\frac{\tau}{2} k \leq\sum_{i=1}^k\kappa(\Delta_i)\leq
	2\pi S+\frac{\tau}{2} k .
	\label{2.3a}
\end{equation}

\item
$-\int _{\Delta_i}K >3\pi,$ and so,  
\begin{equation} \label{2.4}
-\int _{\cup_{i=1}^k\Delta_i}K
=-2\pi\chi(\cup_{i=1}^k\Delta_i)+\int _{\cup_{i=1}^k\partial \Delta_i}\kappa_g>3k\pi.
\end{equation}
\end{enumerate}

\item
\label{it6}
\underline{Genus estimate outside the concentration of curvature:}
If $M$ is orientable, $k\geq1 $ and the genus $g(M)$ of $M$ is finite, then the
genus $g(\widetilde{M})$ of $\widetilde{M}$ satisfies
$0\leq g(M)-g(\wt{M})\leq 	3I-2$.
		
\item
\label{it8}
\underline{Area estimate outside the concentration of curvature:}
If $k\geq 1$, then
\[
\mbox{\rm Area}(\widetilde{M})
\geq 2\pi \sum_{i=1}^k m(i)r_F(i)^2
\geq\mbox{\rm Area}\left( \bigcup_{i=1}^k\Delta_i\right) \geq
k\pi \de_1^2.
\]
	
\item
\label{itF}
There exists a $C>0$,  depending   on $\ve_0,K_0,H_0$ and independent of
$I$, such that
\begin{equation}
\mbox{\rm Area}(M)\geq
\left\{
\begin{array}{ll}
{\displaystyle C\max\{1,\mbox{\rm Radius}(M)\} }&
\mbox{if $\partial M\neq \varnothing$,}
\\
{\displaystyle C\max\{1,\mbox{\rm Diameter}(M)\} } & \mbox{if $\partial
M=\varnothing$,}
\end{array}\right.
\label{Area}
\end{equation}
where
\begin{eqnarray}
\mbox{\rm Radius}(M)&=&\sup_{x\in M} d_M(x,\partial M)\in (0,\infty]
\quad \mbox{if }\partial M\neq \varnothing,\nonumber
\\
\mbox{\rm Diameter}(M)&=&\sup_{x,y\in M} d_M(x,y) \qquad \qquad \quad \;
\mbox{if }\partial M=\varnothing.\nonumber	
\end{eqnarray}
In particular, if $M$ has infinite radius
or if $M$ has empty boundary and it is non-compact, then
its area is infinite.
\end{enumerate}
\end{theorem}

\begin{definition} \label{DefIndexNO}
{\rm
Given a $1$-sided minimal immersion $F\colon M\looparrowright X$,
let $\wt{M}\to M$ be the two-sided cover of
$M$ and let $\tau \colon \wt{M}\to \wt{M}$ be the associated deck
transformation of order 2. Denote by $\wt{\Delta}$,
$|\wt{A}|^2$ the Laplacian and squared norm of the second fundamental
form of $\wt{M}$, and let $N\colon \wt{M}\to TX$
be a unitary normal vector  field.  The index of $F$ is defined as the
number of negative eigenvalues of the elliptic, self-adjoint operator
$\wt{\Delta} +|\wt{A}|^2+\mbox{Ric}(N,N)$ defined over the space of
compactly supported smooth functions $\phi \colon \wt{M}\to \R$
such that $\phi \circ \tau =-\phi $.
}
\end{definition}

\begin{definition} \label{defharm}
{\rm
Given a (smooth) Riemannian manifold $X$, a local chart $(x_1,\ldots x_n)$
defined on an open set $U$ of $X$
is called {\it harmonic} if $\Delta x_i=0$ for all $i=1,\ldots n$.
		
Following Definition 5 in~\cite{hehe}, we make the next definition.
Given $Q>1$ and $\alpha \in (0,1)$, we define the $C^{1,\alpha }$-{\it harmonic
radius} at a point $x_0\in X$ as the largest number
$r=r(Q,\a)(x_0)$ so that on the geodesic ball $B_X(x_0,r)$ of center $x_0$
and radius $r$, there is a harmonic coordinate chart
such that the metric tensor $g$ of $X$ is $C^{1,\a}$-controlled in these
coordinates. Namely, if $g_{ij}$, $i,j=1,\ldots ,n$, are
the components of $g$ in these coordinates, then
\begin{enumerate}
\item $Q^{-1}\de _{ij}\leq g_{ij}\leq Q\, \de_{ij}$ as bilinear forms,
\item ${\displaystyle \sum_{\be =1}^3 r \sup _{y}|\frac{\partial g_{ij}}{\partial x_{\be}}(y)|
+\sum _{\be =1}^3r^{1+\a}\sup _{y\neq z}\frac{\left|
\frac{\partial g_{ij}}{\partial x_{\be}}(y)-
\frac{\partial g_{ij}}{\partial x_{\be}}(z)\right| }{d_X(y,z)^{\a}}}\leq Q-1$.
\end{enumerate}
The $C^{1,\a}$-harmonic radius $r(Q,\a )(X)$ of $X$ is now defined by
\[
r(Q,\a)(X)=\inf_{x_0\in X}r(Q,\a)(x_0).
\]
If the absolute sectional curvature of $X$ is bounded by some
constant $K_0>0$ and Inj$(X)\geq r_0>0$, then Theorem~6 in~\cite{hehe}
implies that given $Q>1$ and $\a \in (0,1)$
there exists $C=C(Q,\a ,r_0,K_0)$ (observe that $C$ does not depend
on $X$) such that $r(Q,\a)(X)\geq C$.
	}
\end{definition}

\begin{definition} \label{DefMulti}
{\rm Let $f\colon\S\looparrowright \R^3$ be an immersed annulus,  $P$
a plane passing through the origin and 
$\Pi\colon \rth \to P$ the orthogonal projection.
Given $m\in \N$, let $\sigma_m\colon P_m \to P^*=P\setminus \{\vec{0}\}$  be the $m$-sheeted
covering space of $P^*$.  We say that $\S$ is an $m$-{\it valued graph}
over $P$ if $\vec{0}\not \in (\Pi\circ f)(\S)$
and $\Pi\circ f\colon \S \to P^*$
has a smooth injective lift $\wt{f} \colon  \S \to P_m$
through $\sigma_m$; in this case, we say that $\Sigma $ has {\it degree $m$}
as a multi-graph.

Given $Q>1$ and $\a \in (0,1)$, let $X$ be a Riemannian 3-manifold
and $(x_1,x_2,x_3)$ a harmonic chart for $X$ defined on $B_X(x_0,r)$, $x_0\in X$, $r>0$,
where the metric tensor $g$ of $X$ is $C^{1,\a}$-controlled in the
sense of Definition~\ref{defharm}. Let
$P\subset B_X(x_0,r)$ be the image by this harmonic chart of the
intersection of a plane in $\R^3$ passing through the origin with the domain of the chart.
In this setting, the notion of $m$-valued graph over $P$ generalizes
naturally to an immersed annulus $f\colon\S\looparrowright B_X(x_0,r)$,
where the projection $\Pi$ refers to the harmonic coordinates.
If $f\colon\S\looparrowright B_X(x_0,r)$ is an $m$-valued graph over $P$
and $u$ is the corresponding graphing function that
expresses $f(\Sigma)$, we can consider the gradient $\nabla u$ with
respect to the metric on $P$ induced by the ambient metric of $X$.
Both $u$ and $|\nabla u|$ depend on the choice of harmonic coordinates
around $x_0$ (and they also depend on $Q$),
but if $\frac{|u(x)|}{|x|}+|\nabla u|<\tau $ for some $\tau \in (0,\pi /10]$
and $Q>1$ sufficiently close to $1$, then
$\frac{|u(x)|}{|x|}+|\nabla u|<2\tau$ for any other choice of harmonic
chart around $x_0$ with this restriction of $Q$.
}
\end{definition}

\section{The proof of Theorem~\ref{main2}} \label{sec:Area}

This section is dedicated to the proof of Theorem~\ref{main2}.
Note the complete surfaces considered in this theorem have empty
boundary. Let $F\colon M\looparrowright X$ be an immersion as in the
statement of Theorem~\ref{main2}.

We will use the notation in Theorem~\ref{mainStructure} and fix $\tau
=\pi /10$.
Notice that as the boundary of $M$ is empty, then we may consider
the $H$-immersion $F\colon M\la X$ of index $I$ described in
Theorem~\ref{main2}
to be an element of $\L(I, H_0,r_0,1,K_0)$, where
$H_0,r_0,K_0$ are given in the hypotheses of Theorem~\ref{main2}.

\subsection{Normalizing the space $\Lambda$}
\label{sec3.1}
After scaling the Riemannian metric of $X$ by the square root of
\begin{equation}
	\label{3.0}
\lambda=\max\{1,\frac{1}{r_0},\sqrt{K_0}, H_0\},
\end{equation}
one obtains a new
Riemannian manifold $X'$;
note that this scaling of the metric scales arc length in $X$ by the
factor $\lambda\geq 1$, and that the
metric of $M$ induced by the isometric immersion $F$ creates an
associated isometric immersion $F'\colon M\la X'$ such
that $F'(p)=F(p)$ for each $p\in M$.
After this homothetic change of the metric,
we can consider  $F'\colon M\la X'$ to be an immersion satisfying the
following properties:
\ben
\item $\Inj(X')\geq 1$.
\item The absolute sectional curvature of $X'$ is less than or equal to 1.
\item $F'$ is an isometric immersion of constant mean curvature $H'
\in [0,1]$.
\item $(F'\colon M\la X')\in \L(I, 1,1,1,1)$.
\item   $\mbox{Area}(F)=\lambda^2 \,\mbox{Area}(F').$
\item  $\mbox{Diameter}(F)=\lambda \,\mbox{Diameter}(F').$
\een
Items 5 and 6 above allow us to easily convert estimates on the
area of subdomains and lengths of curves in the domain of $F$
to areas and lengths of the corresponding domains and curves in the
domain of $F'$, and thereby,
these conversion formulae reduce the proofs of statements given in
Theorem~\ref{main2} for $F\in \L(I, H_0,r_0,1,K_0)$ to
the corresponding estimates for $F'$ in $\L(I, 1,1,1,1)$.
Thus,
for the remainder of the proof of Theorem~\ref{main2}, we will assume
$F\colon M\la X$ lies in $\L(I, 1,1,1,1)$, and refer to Area$(M)$,
Diameter$(M)$ for those with respect to the induced metric by $F$.

\subsection{Proof of item~\ref{It0} of Theorem~\ref{main2}}
\label{sec3.2}
Consider an element $(F\colon M\la X)\in \L(I, 1,1,1,1)$.
If $M$ is non-compact, then the last sentence in item~\ref{itF} of
Theorem~\ref{mainStructure} states that $M$ has infinite area,
which proves that the inequality Area$(M)\geq C_A$
in item~\ref{It0} holds vacuously (for any choice of $C_A>0$).
If moreover  $M$ is connected, then there exists
a geodesic ray in $M$, i.e., an embedded, length-minimizing unit-speed
geodesic arc $\g\colon[0,\infty)\to M$; in particular, the diameter of
$M$ is infinite, and thus the second statement in item~\ref{It0} also
holds vacuously.

For the remainder of this section we will assume that $M$ is compact.

\begin{lemma}
\label{lemma3.1}
Given $x_0\in M$, let $M(x_0)$ be the component
 of $M$ containing $x_0$. Then, $M(x_0)$ is not contained in
the closed extrinsic ball $\overline{B}_X(x_0,\pi/4)$ (in particular,
$\partial B_M(x_0,\pi/4)$ is not empty).
\end{lemma}
\begin{remark}
{\rm
Observe that if the lemma holds,  then the extrinsic diameter of $M$
is greater than $\pi/4$ (in particular, the intrinsic diameter has the same lower bound),
which proves the second statement in item~\ref{It0} of
Theorem~\ref{main2}.
}
\end{remark}
\begin{proof}[Proof of Lemma~\ref{lemma3.1}]
Fix a point $x_0\in M$ and let $r\in (0,\pi/4)$. Since the
injectivity radius of $X$ is at least 1, all the distance spheres
$\partial B_X(x_0,r)$ with $r\in (0,1)$ are geodesic spheres.
By comparison results and since the absolute sectional curvature of $X$
is bounded by 1, the second fundamental form of
$\partial B_X(x_0,r)$ has normal curvatures greater than $1$.
Assume that $M(x_0)$ is contained in $\ov{B}_X(x_0,r)$. As $M(x_0)$
is compact, then there exists a largest $r_1\in (0,r]$ such that
$M(x_0)\subset \overline{B}_X(x_0,r_1)$, and there exists $x\in
M(x_0)\cap \partial B_X(x_0,r_1)$. This implies that all the normal
curvatures of $M$ at $x$ are greater than $1$, which implies that
the mean curvature of $M$ is greater than 1, which contradicts
that $F\colon M\la X$ lies in $\Lambda (I,1,1,1,1)$.
This contradiction proves that $M(x_0)$ cannot be contained in
$\overline{B}_X(x_0,r)$. Since this holds for every $r\in (0,\pi/4)$
and $M$ is compact, we conclude that $M(x_0)$ cannot be contained in
$B_X(x_0,\pi/4)$. In fact, $M(x_0)$ cannot be contained in $\overline{B}_X(x_0,\pi/4)$
(otherwise the maximum principle for the mean curvature operator would imply that
$M(x_0)=\partial B_X(x_0,r)$, which contradicts that $x_0\in M(x_0)$).
Now the lemma is proved.
\end{proof}

Using~\cite[Proposition~2.5 and item~3 of Remark~2.5]{mpe20} with $R_1=a=H_0=1$,
for each $p\in \Int(M)$ we have
\begin{equation}
\label{yaulemma0}
\mbox{\rm Area}[B_M(p,r)]\geq E(r):=\pi\, r^2 e^{-2r-1+r\cot(r)}
\quad \mbox{for every }r\in (0,\pi/4].
\end{equation}
Therefore, since $M\not\subset \ov{B}_X(p,\pi/4)$ by
Lemma~\ref{lemma3.1}, the extrinsic diameter of $M$ is greater than
$\pi/4$ and $\mbox{Area}(M)>
\mbox{Area}[B_M(p,\pi/4)]=E(\pi/4)\approx 0.325043$.
 This completes the proof of item~\ref{It0} of Theorem~\ref{main2}.

\subsection{Proof of item~\ref{It1a} of Theorem~\ref{main2}}
\label{sec3.3}
Consider an element $(F\colon M\la X)\in \L(I, 1,1,1,1)$.
If $M$ is non-compact, then the last sentence in item~\ref{itF} of
Theorem~\ref{mainStructure} states that $M$ has infinite area,
which vacuously implies item~\ref{It1a} of the theorem holds
(for any choice of $C_1(I)$). Henceforth, assume $M$ is compact.

Let $M=M_1\cup \ldots \cup M_b$, $b\in \N$, be the decomposition 
of $M$ in connected components. Assume inequality~\eqref{1.1} holds for each
$M_i$ with respect to a constant $C_1=C_1(I)$. Since the index of
each $M_i$ is at most $I$, then 
\begin{equation}\label{3.3b}
\mbox{Area}(M)=\sum_{i=1}^b\mbox{Area}(M_i)\geq 
\sum_{i=1}^bC_{1}(g(M_i)+1)=C_{1}(g(M)+b)\geq C_{1}(g(M)+1)
\end{equation}
where $g(M_i)$ is the genus of $M_i$.
Hence, it suffices to prove that~\eqref{1.1} holds under
the additional assumption that $M$ is connected, which we will assume henceforth.

The region  $\widetilde{M}\subset M$ defined in item~\ref{it7} of
Theorem~\ref{mainStructure} for the space $\L(I, 1,1,1,1)$ produces
a uniform bound $A_1=A_1(I)\geq 1$ from above on
the norm the second fundamental form of $\wt{M}$. Let us define \begin{equation}
\label{3.3a}
K_1=K_1(I):=-1-\frac{1}{2}A_1^2.
\end{equation}
Since $A_1\geq 1$,  then  $K_1\leq -\frac32$.
The Gauss equation gives
\begin{equation}
K=K_X(TM)+\det(A_M),
\label{3.1}
\end{equation}
where $K$ denotes the Gaussian curvature of $M$ and $K_X(TM)$
is the sectional curvature of $X$ for the tangent plane to $M$.
Since the absolute sectional curvature of $X$ is bounded by 1,
$H^2\geq \det(A)$ and $H\in[0,1]$, we have the following
upper and lower estimates for $K$ in $\wt{M}$:
\begin{equation}
K_1\leq -1-\frac{1}{2}|A_M|^2\leq -1+\det(A_M)\leq K\leq 1+\det(A_M)
\leq 1+H^2\leq 2.
\label{3.2}
\end{equation}

\subsubsection{Item~\ref{It1a} holds when $k=0$.}
We first show that item~\ref{It1a}
of Theorem~\ref{main2} holds in the special case that the integer $k$
defined in Theorem~\ref{mainStructure} is zero. To see
this, observe that $\wt{M}=M$, and thus, \eqref{3.2} ensures that
$M$  has Gaussian curvature
bounded from below by $K_1$ and from above by 2.
Let $\wh{M}$ be the orientable cover of $M$.

Suppose $g(M)=0$ (recall that $M$ was assumed to be compact and
connected). Applying to $\wh{M}$ the Gauss-Bonnet theorem, we have
\[
2\cdot \mbox{Area}(\wh{M})\geq \int_{\wh{M}}K =4\pi.
\]

If $M$ is non-orientable, then $\mbox{Area}(M)=
\frac{1}{2}\mbox{Area}(\wh{M})\geq \pi$. This
inequality also holds in the case $M$ is orientable (in fact,
$M=\wh{M}$ and so, $\mbox{Area}(M)=\mbox{Area}(\wh{M})\geq 2\pi$).
Therefore, inequality~\eqref{1.1} holds with $C_1(I)=\pi$ if $g(M)=0$ and
$k=0$.

Suppose now that $g(M)\geq 2$.
Hence, Gauss-Bonnet applied to $\wh{M}$ gives
\[
-K_1\cdot \mbox{Area}(\wh{M})\geq -\int_{\wh{M}}K
=-2\pi \chi(\wh{M})=4\pi(g(\wh{M})-1)=4\pi(g(M)-1)\geq \frac{4\pi}{3}(g(M)+1).
\]

If $M$ is non-orientable, then $\mbox{Area}(M)=
\frac{1}{2}\mbox{Area}(\wh{M})\geq \frac{2\pi}{3|K_1|}(g(M)+1)$.
This inequality also holds in the case $M$ is orientable (in fact,
$M=\wh{M}$ and thus, $\mbox{Area}(M)=\mbox{Area}(\wh{M})\geq
\frac{4\pi}{3|K_1|}(g(M)+1)$). Therefore, inequality~\eqref{1.1} holds
with $C_1(I)=\frac{2\pi}{3|K_1|}$ if $g(M)\geq 2$ and $k=0$.

By the already proven item~\ref{It0} of Theorem~\ref{main2}, the area of $M$ is at least $C_A$.
In particular
if $g(M)=1$ (i.e., $M$ is a torus or a Klein bottle), then one can still
obtain a lower bound estimate for the area of $M$ by
$$\mbox{Area}(M)\geq  C_A=\frac{C_A}{2}(g(M)+1).$$
Therefore, inequality~\eqref{1.1} holds with $\ds C_1(I)=\frac{C_A}{2}$
if $g(M)=1$ and $k=0$.

Finally we consider the minimum of the constants
$\pi, \frac{2\pi}{3|K_1|}, \frac{C_A}{2}$ obtained
in the three cases above. As observed previously,  $|K_1|\geq \frac32$, and so,
\[
C_3=C_3(I)= \min\{\pi, \frac{2\pi}{3|K_1|}, \frac{C_A}{2}\}= \min\{\frac{2\pi}{3|K_1|}, \frac{C_A}{2}\},
\]
we deduce that~\eqref{1.1} holds with $C_3$, instead of $C_1$, for connected compact $M$ when $k=0$.

\subsubsection{Item~\ref{It1a} holds when $k\geq 1$.}

Assume that  $k\geq 1$ (in particular, $I\geq 1$)
and we will  obtain a constant
$C_4=C_4(I)\in (0, \pi \de_1^2)$ that satisfies
\begin{equation}
\mbox{Area}(M)\geq C_4(g(M)+1),
\label{3.3}
\end{equation}
which will complete the proof of
item~\ref{It1a} of Theorem~\ref{main2}
after setting $C_1(I)=\min\{C_3(I),C_4(I)\}$.
We will need the following two claims.

\begin{claim}
	\label{claim3.1}
If $g(M)\geq 12I-3$, then inequality~\eqref{3.3} holds with constant $C'_4(I)=\frac{\pi}{|K_1(I)|}$.
\end{claim}
\begin{claim}
\label{claim3.2}
If  $g(M)< 12I-3$. then inequality~\eqref{3.3} holds with $C''_4(I)=\frac{C_A}{12I-3}$. 
\end{claim}
\begin{proof}[Proof of Claim~\ref{claim3.1}]
We start applying the Gauss-Bonnet formula and~\eqref{3.2}:
\begin{equation}
|K_1|\cdot \mbox{\rm Area}(\widetilde{M})\geq
\left| \int_{\widetilde{M}}K\right| =
\left|\int _{{M}}K -\int _{\cup_{i=1}^k\Delta_i}K\right|.
\label{3.4}
\end{equation}
On the other hand, calling $g=g(M)$,
\begin{eqnarray}
	\int _{{M}}K -\int _{\cup_{i=1}^k\Delta_i}K
	&= &
	2\pi (\chi({M})-2\pi\chi(\cup_{i=1}^k\Delta_i)
	+\int_{\partial (\cup_{i=1}^k\Delta_i)}  \kappa_g	
	\nonumber
	\\
	&\leq& 2\pi (1-g)+2\pi(6I-2S-e) +2\pi S+\tau k,
	\label{3.5}
\end{eqnarray}
where in the last inequality we have used item~\ref{it4}(d) of
Theorem~\ref{mainStructure} and~\eqref{2.3a}.
Since $\tau\leq \pi/10$
in Theorem~\ref{mainStructure} and $S+e\geq 4k$ (this last inequality
follows since $e(i)\geq 1$ and $m(i)\geq 2$, and if $e(i)=1$
then $m(i)\geq 3$), we can bound~\eqref{3.5} from above by
$2\pi (1-g+6I-4k+\frac{k}{20})$, which in turn is at most
$2\pi (-g+6I-2)$ because $k\geq 1$. Therefore,
\begin{equation}
	\int _{{M}}K -\int _{\cup_{i=1}^k\Delta_i}K
	\leq -\pi (2g-12I+4).
	\label{3.6}
\end{equation}
Since $g\geq 12I-3$ by hypothesis, then the RHS of~\eqref{3.6} is at most
$-\pi (g+1)$, and thus, we conclude that
\begin{equation}
\int _{{M}}K -\int _{\cup_{i=1}^k\Delta_i}K\leq -\pi (g+1).
\label{3.7}
\end{equation}
Now, \eqref{3.7} and \eqref{3.4} give
\begin{equation}
|K_1|\cdot \mbox{\rm Area}(\widetilde{M})\geq \pi(g+1),
\label{3.9}
\end{equation}
from where Claim~\ref{claim3.1} follows.
\end{proof}

\begin{proof}[Proof of Claim~\ref{claim3.2}]
By the already proven item~\ref{It0} of Theorem~\ref{main2}, we have
$\mbox{Area}(M)\geq C_A$,
which is $\geq \frac{C_A}{12I-3}(g(M)+1)$ 
$g(M)<12I-3$. This finishes the proof of Claim~~\ref{claim3.2}.
\end{proof}

Once Claims~\ref{claim3.1} and \ref{claim3.2} are proved, we will
conclude that inequality~\eqref{3.3} holds in all cases with $k\geq 1$
with $C_4(I)=\min\{C'_4(I),C''_4(I)\}$. This completes the proof of
item~\ref{It1a} of Theorem~\ref{main2}.


\subsection{A preliminary result on area estimates of balls in $M$ and its diameter}
\label{sec3.4}
We temporarily pause the proof of Theorem~\ref{main2} to state and prove the next auxiliary  result, which gives
general upper estimates on the areas of balls of radius $r$ in $M$ and general upper estimates
on the diameter of $M$ in terms of constants described
in the Hierarchy Structure Theorem~\ref{mainStructure}.  The next theorem will be crucial in the proofs of the remaining items 2 and 3
of Theorem~\ref{main2}. The proof of Theorem~\ref{main22} will be
given in Sections~\ref{sec3.5}, \ref{sec3.6} and \ref{sec3.7}.
\begin{theorem}[Area estimates for intrinsic
	balls and diameter estimates for $M$]
	\label{main22}
	For $r_0>0$, $K_0, H_0\geq 0$, consider all
	complete Riemannian  3-manifolds $X$ with injectivity radius
	$\Inj(X)\geq r_0$ and absolute sectional curvature bounded from above by
	$K_0$, and let $\lambda=\max\{1,\frac{1}{r_0},\sqrt{K_0}, H_0\}$.
	Let $M$ be a complete immersed $H$-surface in $X$ with empty boundary,
	$H\in [0,H_0]$, index at most~$I\in \N\cup \{ 0\}$ and genus $g(M)$,
	which in the language of Theorem~\ref{mainStructure} implies
	$M\in \Lambda =\L(I, H_0, r_0, 1, K_0)$ with additional chosen
	constant $\tau=\pi/10$. Then:
	
	\begin{enumerate}
		\item Suppose that one of the following two conditions holds:
		\begin{enumerate}[(i)]
			\item $I=0$, i.e., $M$ is stable.
			\item $I\geq 1$ and $k=0$ with the notation of
			Theorem~\ref{mainStructure} (in particular, $M=\wt{M}$).
		\end{enumerate}
		Depending on whether condition (i) or (ii) holds, we introduce the
		following constant $K_1=K_1(I)$.
		If condition (i) holds, let  $C_s\geq 2\pi$ be the
		universal curvature estimate for stable $H$-surfaces described in
		Theorem~\ref{stableestim1s} below, and let $K_1:=
		-1-\frac{1}{2} C_s^2$. If condition (ii) holds, let
		$K_1=K_1(I)=-1-\frac{1}{2}A_1^2$ where $A_1=A_1(I)\geq 1$ is the constant
		given by Theorem~\ref{mainStructure} for the space $\Lambda(I,1,1,1,1)$.
		
		For all $x\in M$ and $r>0$,
		\begin{equation}
			\mbox{\rm Area}(B_M(x,r))\leq
			\frac{2\pi}{-K_1\l^2}\left[ \cosh \left( \l \sqrt{-K_1}r\right)
			-1\right],
			\label{1.2a}
		\end{equation}
		and if $M$ is connected, then
		\begin{equation}
			\mbox{\rm Diameter}(M)\geq
			\frac{1}{\sqrt{K_1}\l}\arccosh \left[\frac{-K_1C_1(I)}{2\pi}
			(g(M)+1)+1\right].
			\label{1.5a}
		\end{equation}
		
		\item Suppose $k\geq 1$ (in particular, $I\geq 1$).
		Let $\de=\de(I)\in (0,\frac{1}{2}]$ be the
		constant described in Theorem~\ref{mainStructure} for the space
		$\L(I, 1, 1, 1, 1)$, and let $\Delta_1,\ldots ,\Delta_k$, $k\leq I$,
		be the smooth compact domains associated to $M$ introduced in item~\ref{it1} of
		Theorem~\ref{mainStructure}. There exists $A_3(I) \geq 6$, independent of
		$M,r_0,K_0,H_0$, such that for all $x\in M$  and all $r > 0$, the following estimates hold:
		\begin{equation}
			\mbox{\rm Area}(B_M(x,r)\setminus \cup_{i=1}^k\Delta_i)\leq
			\frac{2(6\pi+1)}{\l^2A_3(I)}I\left[ 2[\cosh(\l \sqrt{A_3(I)}r)-1]
			+\frac{\sinh(\l \sqrt{A_3(I)}r)}{A_3(I)^{1/2}}\right],
			\label{1.4a}
		\end{equation}
		\begin{equation}
			\mbox{\rm Area}(B_M(x,r))\leq
			\frac{I}{\l^2}\left[ \frac{2(6\pi+1)}{A_3(I)}\left(
			2[\cosh(\l \sqrt{A_3(I)}r)-1]
			+\frac{\sinh(\l \sqrt{A_3(I)}r)}{A_3(I)^{1/2}}\right)
			+\frac{3\pi}{8}\right],
			\label{1.4b}
		\end{equation}
		and if $M$ is connected, then
		\begin{equation}
			\mbox{\rm Diameter}(M)\geq
			\frac{1}{\l \sqrt{A_3(I)}}\arccosh \left[\frac{C_1(I)}{20 I}(g(M)+1)\right].
			\label{1.5'}
		\end{equation}
	\end{enumerate}
\end{theorem}

\subsection{The proof of item~1 of Theorem~\ref{main22}}
\label{sec3.5}
It is worth recalling some aspects related to curvature estimates for
complete stable $H$-surfaces  $\S$ (possibly with boundary)
in complete Riemannian 3-manifolds of absolute sectional
curvature at most 1 and injectivity radius at most 1; such curvature
estimates are independent on the value of the (constant) mean curvature.
Rosenberg, Toubiana and Souam~\cite[Main Theorem]{rst1}
proved that  there exists a universal constant
$C'_s>0$  such that if $\S$ is two-sided, then for any point $p\in \S$
of distance at least 1 from $\partial \S$, then $|A_M|(p)\leq C'_s$.
In~\cite{mpe18}, we generalized this curvature estimate to
include the case where $M$ is not necessarily two-sided.
Namely, we proved the following statement.

\begin{theorem}[Curvature estimate for stable $H$-surfaces~\cite{mpe18}]
\label{stableestim1s}
There exists $C''_s\geq 2\pi$ such that given $K_0>0$ and
a complete Riemannian 3-manifold $(Y,g)$ of bounded sectional curvature
$|K|\leq K_0$, then for any immersed
one-sided stable minimal surface $M\la Y$ and for any $p\in M$,
\begin{equation}
\label{eqcurvestim}
|A_{M}|(p)\leq \frac{C''_s}{\min \{ \Inj_Y(p),d_{M}(p,\partial M),
\frac{\pi}{2\sqrt{K_0}}\} }.
\end{equation}

Let $C_s:=\max\{C'_s,C''_s\}$. Given $\ve_0>0$, $K_0\geq 0$, if $X$ is a
complete Riemannian 3-manifold with injectivity radius at
least $\ve_0$ and bounded sectional curvature $|K|\leq K_0$,
and $F\colon M\la X$ is a stable $H$-immersion,
then
\begin{equation}
\label{eqcurvestim2}
|A_{M}|(p)\leq \frac{C_s}{\min \{ \ve_0,d_{M}(p,\partial M),
\frac{\pi}{2\sqrt{K_0}}\} }.
\end{equation}
\end{theorem}

Consider an element $(F\colon M\la X)\in \L(I=0, 1,1,1,1)$, in particular,
$M$ is stable. Particularizing~\eqref{eqcurvestim2} to the case $\ve_0=1$,
$\partial M=\varnothing$, $K_0=1$, we get that
$|A_M|\leq C_s$ in $M$. By the same argument using the Gauss equation
as in~\eqref{3.3a} and \eqref{3.2}, we deduce that the Gaussian
curvature $K$ of $M$ satisfies $K\geq K_1:=-1-\frac{1}{2}C_s^2$
in $M$. In this setting, the Bishop-Cheeger-Gromov relative volume
comparison theorem (see e.g.~\cite[Lemma 36]{pet1}) implies that for every
$x\in M$ and $r>0$,
\begin{equation}\label{3.4a}
\mbox{Area}(B_M(x,r))\leq \mbox{Area}(\B_{K_1}(r))=
\frac{2\pi}{-K_1}\left[ \cosh\left(\sqrt{-K_1}r\right)-1
\right],
\end{equation}
where $\B_{K_1}(r)$ denotes the metric ball of radius $r$ in
the hyperbolic plane of curvature $K_1<0$. This finishes the proof
of inequality~\eqref{1.2a} provided that $I=0$ (that is assuming
condition 1(i) in Theorem~\ref{main22} holds).

Now suppose $(F\colon M\la X)\in \L(I, 1,1,1,1)$, $I\geq 1$.
To prove~\eqref{1.2a} provided that condition 1(ii) in 
Theorem~\ref{main22}
holds, observe that in this case $M=\widetilde{M}$. By~\eqref{3.2},
$K\geq K_1$ where $K_1=-1-\frac{1}{2}A_1^2$
and $A_1=A_1(I)\geq 1$ is given by Theorem~\ref{mainStructure} for
the space $\Lambda(I,1,1,1,1)$. Applying the above arguments to this
new choice of the constant $K_1$, we get that~\eqref{3.4a} holds,
which proves \eqref{1.2a} provided that 1(ii) in Theorem~\ref{main22}
holds.

In order to show~\eqref{1.5a} (regardless of whether
condition 1(i) or 1(ii) in Theorem~\ref{main22} holds),
assume $M$ is connected. Observe that we
can assume that $M$ is compact (otherwise its diameter is infinite by
the argument in the first paragraph of Section~\ref{sec3.2}
and~\eqref{1.5a} holds vacuously). Choose a point $x\in M$.
Taking $r=\mbox{Diameter}(M):=D$ in~\eqref{3.4a} and using
the already proven inequality~\eqref{1.1}, we get
\[
C_1(I)(g(M)+1)\leq \mbox{Area}(M)=\mbox{Area}(B_M(x,D))
\stackrel{\eqref{3.4a}}{\leq}
\frac{2\pi}{-K_1}\left[ \cosh\left(\sqrt{-K_1}D\right)-1
\right],
\]
or equivalently,
\[
D\geq \frac{1}{\sqrt{K_1}}\arccosh \left[\frac{-K_1C_1(I)}{2\pi}
(g(M)+1)+1\right],
\]
which proves inequality~\eqref{1.5a}, and so, finishes the proof of
item~1 of Theorem~\ref{main22}.

\subsection{Area growth of  collar neighborhoods of $\wt{M}$ if $k\geq 1$}
\label{sec3.6}

\begin{definition}
{\rm
For a complete surface $\S$ with boundary $\partial \S$ and for any
$r>0$, let
\[
\S(r)=\{x\in \S\mid d_\S(x,\partial \S)\leq r\}
\]
be the collar neighborhood of $\partial \S$ in $\S$ of radius $r$.}
\end{definition}

Consider an element $(F\colon M\la X)\in \L(I, 1,1,1,1)$.
Assume $k\geq 1$ with the notation of Theorem~\ref{mainStructure}.
Hence, $\wt{M}$ is a surface with smooth boundary.
For later uses, next we will give an upper estimate for the area growth
of the collar neighborhood $\wt{M}(r)$ of $\wt{M}$, $r>0$.

\begin{proposition} \label{prop3.11}
Let $c_1,\ldots,c_e$ be the set of components of
$\partial \wt{M}$. Choose for each $i\in \{1,\ldots,e\}$ a
parametrization by arc length $\g_i\colon [0,L_i] \to c_i$
with associated geodesic curvature function
$\kappa_i(t)$ with respect to the inward pointing unit conormal
vector $\eta=\eta(t)$ of $\wt{M}$ along $\partial \wt{M}$.
Then, for each $i\in \{1,\ldots,e\}$:
\begin{enumerate}
\item $\kappa_i(t)$ is negative in $[0,L_i]$.
\item There exists a complete annulus $\S_i$ with boundary,
with constant Gaussian curvature $K_1$ (this constant is defined
in~\eqref{3.3a}), whose boundary is parameterized by arc length by
$\wh{\g}_i=\wh{\g}_i(t)
\colon [0,L_i]\to \partial \S_i$, and such that the geodesic curvature
function $\wh{\kappa}_i$ of $\partial \S_i$ with respect to the inward
pointing unit conormal vector of $\S_i$ along $\partial \S_i$ satisfies
$\wh{\kappa}_i(t)=\kappa_i(t)$ for all $t\in [0,L_i]$.
Furthermore, $\S_i$ is unique up to isometry.
\item \label{it3prop3.7}
For each $r>0$ we have
\newline
\begin{equation}
\mbox{\rm Area}(\ov{}\wt{M}(r))\leq \sum_{i=1}^e\mbox{\rm Area}(\S_i(r))
=\kappa(\wt{M})\frac{1-\cosh(\sqrt{-K_1}r)}{-K_1}
+L\frac{\sinh(\sqrt{-K_1}r)}{(-K_1)^{3/2}}.
\label{3.14a}
\end{equation}
where $\kappa_g(t)=\sum_{i=1}^e\kappa_i(t)$  and $\kappa(\wt{M})=\int_{\partial \wt{M}}\kappa_g$  
is the total geodesic curvature of $\partial M$ with respect to the
inward pointing unit conormal vector of $\wt{M}$ along $\partial \wt{M}$, and
$L=\sum_{i=1}^eL_i$ is the length of $\partial \wt{M}$.
\end{enumerate}
\end{proposition}
\begin{proof}
Item~1 follows from the proof of the Hierarchy Structure Theorem~\ref{mainStructure} in~\cite{mpe18};
specifically see Lemma~6.4 in~\cite{mpe18}.
		
Item~2 follows directly from the following two facts. First,
given $\ell>0$ and a smooth  function
$\kappa\colon [0,\ell]\to (-\infty,0)$, standard geometry of curves
in the hyperbolic plane $\H^2(K_1)$ with constant Gaussian curvature $K_1$
ensures that there exists a smooth unit speed curve
$\a\colon [0,\ell]\to \H^2(-K_1)$ such that $\kappa(t)$ is the
geodesic curvature of $\a$ at $\a(t)$, for all $t\in [0,\ell]$;
furthermore, $\a$ is unique up to isometries of $\H^2(K_1)$.
Second, if $n\colon [0,\ell]\to U\H^2(K_1)$ is the unit normal vector
to $\a$ pointing to its non-convex side (here $U\H^2(K_1)$ denotes
the unit tangent bundle to $\H^2(K_1)$), then the map
$\phi\colon [0,\ell]\times [0,\infty)\to \H^2(K_1)$
\begin{equation}
\label{3.15a}
\phi (t,r)=\exp _{\a(t)}(r\, n(t)),\quad (t,r)\in [0,\ell]\times
[0,\infty)
\end{equation}
is a submersion, where $\exp \colon U\H^2(K_1)\to \H^2(K_1)$ is the
exponential map. $\phi$ induces a hyperbolic metric $g_h$ on
$[0,L]\times[0,\infty)$ so that for each $t_0\in [0,\ell]$, the
curve of the form $r\in [0,\infty )\mapsto \phi(t_0,r)$ is a unitary
geodesic orthogonal to the arc $\a([0,L])$ at $\a(t_0)$;
in particular, after identifying the two geodesic
arcs $\phi(\{ 0\}\times [0,\infty))$ and
$\phi(\{ \ell \}\times [0,\infty))$ by a hyperbolic isometry,
we obtain a quotient hyperbolic annulus $(\S_{\a},g_h)$
with the properties desired in item~2, in the special case that
$\ell =L_i$ and $\kappa =\kappa _i$.
		
To prove item~\ref{it3prop3.7} of the proposition,
first observe that since $K\geq -K_1$ on $\wt{M}(r)$ by~\eqref{3.2},
we can use relative volume comparison arguments
see e.g.~\cite[Lemma 36]{pet1}) to deduce that
\[
\mbox{Area}(\wt{M}(r))\leq \sum_{i=1}^e\mbox{Area}(\S_i(r)).
\]
It remains to prove that for all $i=1,\ldots ,e$ and $r>0$, the following holds
\begin{equation}
\label{3.17a}
\mbox{Area}(\S_i(r))=\frac{1}{-K_1}
\left[ (1-\cosh(\sqrt{-K_1}r))\int_0^{L_i}\kappa(s)\, ds
+\frac{L_i}{\sqrt{-K_1}} \sinh(\sqrt{-K_1}r)\right].
\end{equation}
\begin{claim}\label{claim3.8a}
Let $\a\colon [0,\ell]\to \H^2(-1)$ a smooth arc parameterized by arc
length, with negative geodesic curvature function $\kappa=\kappa(s)$.
Consider the complete hyperbolic annulus with boundary
$(\S_{\a},g_h)$ constructed in~\eqref{3.15a} in terms of $\a$
with $K_1=-1$.
Given $r>0$, let $\a_r\colon [0,\ell]\to \S_{\a}$ be the equidistant arc
to $\a$ at distance $r$ on the non-convex side of $\a$. Then:
\begin{enumerate}
\item The geodesic curvature function $\kappa_r$ of $\a_r$ is given by
\begin{equation}\label{3.18c}
\kappa_r(s)=\frac{\kappa(s)-\tanh(r)}{1-\tanh(r)\kappa(s)},
\quad \forall s\in [0,\ell].
\end{equation}
\item Let $\S_{\a}(r)\subset
\S_{\a}$ be the domain enclosed by $\a,\a_r$ and the two
geodesics of $\S_{\a}$ that join the extrema of
$\a,\a_r$ (so that these geodesics are orthogonal to both $\a,\a_r$ at
their extrema). Then,
\begin{equation}\label{3.19j}
\mbox{\rm Area}(\S_{\a}(r))=(1-\cosh(r))\int_0^{\ell}\kappa(s)\, ds+
\ell \sinh(r),
\end{equation}
where $\ell=L(\a)$ is the length of $\a$.
\item If we replace $\H^2(-1)$ by $\H^2(K_1)$, then~\eqref{3.19j}
becomes
\begin{equation}\label{3.20b}
\mbox{\rm Area}(\S_{\a}(r))=\frac{1}{-K_1}
\left[ (1-\cosh(\sqrt{-K_1}r))\int_0^{\ell}\kappa(s)\, ds
+\frac{\ell}{\sqrt{-K_1}} \sinh(\sqrt{-K_1}r)\right].
\end{equation}
\end{enumerate}
\end{claim}
\begin{proof}[Proof of the claim]
Recall that $\S_{\a}$ submerses into $\Hip^2(-1)$
through the map $\phi$ given in~\eqref{3.15a}. In particular, $\S_{\a}
\setminus \partial \S_{\a}$ is locally isometric to $\Hip^1(-1)$.
This property clearly allows us to prove the claim assuming that
$\S_{\a}(r)$ embeds into $\Hip^2(-1)$:
for item~1 of the claim this is obvious, while for
items~2 and 3 we can divide $[0,\ell]$ into a partition
$0=s_0<s_1<\ldots <s_n=\ell$ such that if we denote by
$\a_i=\a|_{[s_{i-1},s_i]}$, $i=1,\ldots ,n$ and we apply the same
procedure as with $\S_{\a}$ to construct $n$ ``rectangles'' $\S_{\a_i}(r)$, then each $\S_{\a_i}(r)$ embeds into $\Hip^2(-1)$.
In this way, both equations~\eqref{3.19j} and~\eqref{3.20b}
will follow by adding up the corresponding equalities over the
rectangles $\S_{\a_1}(r),\ldots ,\S_{\a_n}(r)$, which only intersect
along geodesics in their boundaries. Therefore,
for the remainder of this proof we will assume that $\S_{\a}(r)$ is
embedded in $\Hip^2(-1)$.

We will use the model of $\Hip^2(-1)$ as the upper sheet of a hyperboloid
in the Lorentz-Minkowski space $\LM^3=(\R^3,\langle ,\rangle_L
=dx_1^2+dx_2^2-dx_3^2)$. In this model, $\Hip^2(-1)=
\{ x\in \LM^3\ | \ \langle x,x\rangle_L=-1,\ x_3>0\}$, and the induced
metric by $\langle ,\rangle_L$ on $\Hip^2(-1)$ is positive definite and
has constant Gaussian curvature $-1$. Given
$x\in \Hip^2(-1)$, the tangent plane $T_x\Hip^3$ identifies to
$\langle x\rangle ^{\perp}\subset \LM^3$. Given $x\in \Hip^2(-1)$
and $v\in \langle x\rangle ^{\perp}$, the unique geodesic in $\Hip^2(-1)$
with initial conditions $\g(0)=x$, $\g'(0)=v$ is
\begin{equation}
\label{3.18b}
\g(t)=\g(t,x,v)=\cosh(|v|t))x+\frac{\sinh(|v|t)}{|v|}v,
\end{equation}
and the parallel transport along $\g(\cdot, x,v)$ from $0$ to $t$ is
given by
\begin{equation}
\label{3.18a}
\tau_0^t\colon T_x\Hip^2(-1)=\langle x\rangle ^{\perp}
\to T_{\g(t)}\Hip^2=\langle \g(t)\rangle ^{\perp},
\quad \tau_0^t(w)=w+\frac{\langle v,w\rangle_L}{|v|^2}(\dot{\g}(t)-v),
\end{equation}
where $\dot{\g}(t)=\frac{d\g}{dt}$.
	
Given $s\in [0,\ell]$, the equidistant curve $\a_r$ is
\[
\a_r(s)=\exp_{\a(s)}(r\mathcal{R}\a'(s))=\g(r,\a(s),\mathcal{R}\a'(s))
=\cosh(r)\a(s)+\sinh(r)\mathcal{R}\a'(s),
\]
where $\mathcal{R}$ is the rotation of angle $\pi/2$ in each tangent
plane to $\Hip^2(-1)$ so that $\mathcal{R}\a'$ points to the non-convex
side of $\a$. $\a_r'(s)=J_s(r)$ is the value at $t=r$ of the unique
Jacobi field $J_s=J_s(t)$ along the geodesic
$t\mapsto \g(t,\a(s),J\a'(s))$ with
initial conditions
\begin{equation}
\label{3.19a}
J_s(0)=\a'(s),\quad \frac{DJ_s}{dt}(0)=\frac{D(\mathcal{R}\a')}{ds}(s)
=-\kappa(s)\a'(s),
\end{equation}
where $\{ \a',n_{\a}=\mathcal{R}\a'\}$ is the Frenet dihedron for $\a$.
Since both $J_s(0),\frac{DJ_s}{dt}(0)$ are orthogonal to $\dot{\g}(0)$,
we deduce that $J_s(t)$ is everywhere orthogonal to $\dot{\g}(t)$. In
particular, $J_s(t)=f_s(t)\tau_0^t(\a'(s))$, where $f_s(t)$
is a solution of the ODE $\ddot{f}-f=0$. Imposing \eqref{3.19a},
we have $f_s(0)=1$, $\dot{f_s}(0)=-\kappa(s)$, and thus
\begin{equation}\label{3.21}
f_s(t)=\cosh(t)-\kappa(s)\sinh(t).
\end{equation}
Therefore, an orthogonal basis of the tangent and normal line
to $\a_r$ is
\begin{eqnarray}
\a_r'(s)&=&J_s(r)=f_s(r)\tau_0^r(\a'(s))
\stackrel{\eqref{3.18a}}{=}f_s(r)\a'(s), \hspace{2cm} \mbox{(not
necessarily unitary)}\nonumber
\\
n_{\a_r}(s)&=&\tau_0^r(n_{\a}(s))\stackrel{\eqref{3.18a}}{=}n_{\a}(s)=
\mathcal{R}\dot{\g}(r)=\sinh(r)\a(s)+\cosh(r)\mathcal{R}\a'(s), \quad \mbox{(unitary)}\nonumber
\end{eqnarray}
Hence by the Frenet equations for $\a_r$, we will obtain the negative of
the geodesic curvature $\kappa_r(s)$ of $\a_r$ by taking the derivative
w.r.t. $s$ to $n_{\a_r}(s)$ and dividing by $|\a_r'(s)|=f_s(r)$:
\[
-\kappa_r(s)=\frac{\frac{d}{ds}(n_{\a_r}(s))}{f_s(r)}=
\frac{\sinh(r)-\cosh(r)\kappa(s)}{\cosh(r)-\kappa(s)\sinh(r)},
\]
which proves the first item of the claim.

As for item 2, observe that the interior of
$\S_{\a}(r)$ is topologically a disk,
and that $\partial \S_{\a}(r)$ contains four cusps, in each of which
the exterior angle to $\S_{\a}(r)$ along its boundary is $\pi/2$.
Applying the Gauss-Bonnet Theorem to $\S_{\a}(r)$, we get
\begin{equation}\label{3.23b}
0=-\mbox{Area}(\S_{\a}(r))+\int_{\a}\kappa+\int_{\a_r}\kappa_r
=-\mbox{Area}(\S_{\a}(r))+\int_0^{\ell}\kappa(s)\, ds
-\int_{0}^{\ell}\kappa_r(s)\, ds.
\end{equation}
Using \eqref{3.18c},
\begin{eqnarray}
\int_0^\ell\kappa_r(s)\, ds
&=&\int_{0}^{\ell}\frac{\kappa(s)-\tanh(r)}{1-\tanh(r)\kappa(s)}|\a_r'(s)|
\,ds
\stackrel{\eqref{3.21}}{=}\int_{0}^{\ell}(\cosh(r)\kappa(s)-\sinh(r))\,ds
\nonumber
\\
&=&\cosh(r)\int_0^{\ell}\kappa(s)\, ds-\sinh(r)\ell.\label{3.24'}
\end{eqnarray}
\eqref{3.23b} and \eqref{3.24'} give~\eqref{3.19j}, which finishes
the proof of item~2 of the claim. Item~3 follows from~\eqref{3.19j}
after  an elementary rescaling argument.
\end{proof}
Equation~\eqref{3.17a} follows directly from~\eqref{3.20b} with the
obvious change of notation $\S_i(r)=\S_{\a}(r)$, $L_i=\ell$.
This finishes the proof of Proposition~\ref{prop3.11}.
\end{proof}
	
\subsection{Proof of item 2 of Theorem~\ref{main22}}
\label{sec3.7}
Consider an element $(F\colon M\la X)\in \L(I, 1,1,1,1)$.
Assume $k\geq 1$ (hence $I\geq 1$)
with the notation of Theorem~\ref{mainStructure}.
By~\eqref{3.14a}, we have for each $r>0$
\begin{equation}\label{3.24a}
\mbox{\rm Area}(\ov{}\wt{M}(r))\leq f(r),
\end{equation}
where $f$ is the increasing function
\[
f(r)=\frac{\kappa(\wt{M})}{K_1}[\cosh(\sqrt{-K_1}r)-1]+L\frac{\sinh(\sqrt{-K_1}r)}{(-K_1)^{3/2}}.
\]
\begin{lemma}\label{lema3.8}
Given $x\in M$ and $r>0$, we have
\[
\mbox{\rm Area}\left[ B_M(x,r)\setminus (\cup_{i=1}^k\Delta_i)\right]
\leq f(2r).
\]
\end{lemma}
\begin{proof}
Suppose first that $x\in \cup_{i=1}^k\Delta_i$. Then, $B_M(x,r)\subset \wt{M}(r)\cup
\left( \cup_{i=1}^k\Delta_i\right)$, and thus, $B_M(x,r)\setminus \left( \cup_{i=1}^k\Delta_i\right)
\subset \wt{M}(r)$. Hence,
\[
\mbox{Area}\left[ B_M(x,r)\setminus \left(
\cup_{i=1}^k\Delta_i\right)\right]
\leq \mbox{Area}(\wt{M}(r))\stackrel{\eqref{3.24a}}{\leq} f(r)<f(2r).
\]
Now suppose $x\in \mbox{Int}(\wt{M})$ and let $d>0$ be the distance from
$x$ to $\cup_{i=1}^k\Delta_i$. We distinguish two cases, depending on
whether $r\leq d$ or not.

If $r\leq d$, then since $K\geq K_1$ in $\wt{M}$ and $B_M(x,r)\subset \wh{M}$, the Bishop-Cheeger-Gromov relative volume comparison
theorem implies
\[
\mbox{Area}(B_M(x,r))\leq \mbox{Area}(\B_{K_1}(r))=
\frac{2\pi}{-K_1}\left[ \cosh\left(\sqrt{-K_1}r\right)-1\right]
\stackrel{(\star)}{<} f(r)<f(2r),
\]
where $\B_{K_1}(r)$ denotes the metric ball of radius $r$ in
the hyperbolic plane of curvature $K_1$, and in $(\star)$ we have used
that $|\kappa(\wt{M})|>2\pi$.

If $r>d$, then the triangle inequality ensures that $B_M(x,r)\setminus \left( \cup_{i=1}^k\Delta_i\right)
\subset \wt{M}(2r)$, and thus,
\[
\mbox{Area}\left[ B_M(x,r)\setminus \left( \cup_{i=1}^k\Delta_i\right)\right]
\leq \mbox{Area}(\wt{M}(2r))\stackrel{\eqref{3.24a}}{\leq}f(2r),
\]
which finishes the proof of the lemma.
\end{proof}

Now we are ready to prove inequality~\eqref{1.4a}. First, observe
that~\eqref{2.3a} implies that
\[
-2\pi S-\tau I \leq \kappa(\wt{M})\leq -2\pi S+\tau I,
\]
and so,
\begin{equation}\label{3.25}
0<\frac{\kappa(\wt{M})}{K_1}\leq \frac{2\pi S+\tau I}{-K_1}
\stackrel{(\star)}{\leq}\frac{6\pi+\tau}{-K_1}I\leq \frac{6\pi+1}{-K_1}I,
\end{equation}
where in $(\star)$ we have used that $S\leq 3I$ (this follows from
item~\ref{it4} of Theorem~\ref{mainStructure}).

Second, we can estimate from above the length $L$ of $\partial \wt{M}$
as follows: each component $c_j$ of $\partial \wt{M}$ is contained in
the boundary of a certain compact set $\Delta_i$, and the length
$L(\partial \Delta_i)$ of $\partial \Delta_i$
can be estimated from above
using~\cite[item~(C1) of Lemma~6.1]{mpe18} (also
see~\cite[Remark~6.2]{mpe18}) as
\begin{equation}\label{3.26}
L(\partial \Delta_i)\leq (2 \pi m(i)+1)r_F(i)\leq \frac{2 \pi m(i)+1}
{2}\de,
\end{equation}
where $m(i)$ is the total spinning of the boundary of $\Delta_i$
($m(i)$ was introduced in item~\ref{it4} of Theorem~\ref{mainStructure})
and $\de\leq \ve_0/2=1/2$, $\de_1\leq \frac{\de}{2}$, $r_F(i)\in [\de_1,\de/2]$ were introduced in the main
statement of Theorem~\ref{mainStructure}. Adding up~\eqref{3.26} in the
set $\{ \Delta_1,\ldots ,\Delta_k\}$, we get
\begin{equation}\label{3.27}
L\leq \frac{2 \pi S+k}{2}\de \leq \frac{6 \pi I+I}{2}\de =
\frac{(6 \pi +1)\de }{2}I\leq \frac{6 \pi +1}{4}I
\end{equation}
Finally, Lemma~\ref{lema3.8}, \eqref{3.25} and \eqref{3.27} give
\begin{eqnarray}
\mbox{\rm Area}\left[ B_M(x,r)\setminus (\cup_{i=1}^k\Delta_i)\right]
&\stackrel{\mbox{\footnotesize (Lemma~\ref{lema3.8})}}{\leq }&f(2r)=\frac{\kappa(\wt{M})}{K_1}[\cosh(2\sqrt{-K_1}r)-1]
+L\frac{\sinh(2\sqrt{-K_1}r)}{(-K_1)^{3/2}}
\nonumber
\\
&\stackrel{\eqref{3.25},\eqref{3.27}}{\leq }& \frac{6\pi+1}{-K_1}I[\cosh(2\sqrt{-K_1}r)-1]
+\frac{6 \pi +1}{4}I\frac{\sinh(2\sqrt{-K_1}r)}{(-K_1)^{3/2}}
\nonumber
\\
&=&\frac{2(6\pi+1)}{A_3(I)}I\left[ 2[\cosh(\sqrt{A_3(I)}r)-1]
+\frac{\sinh(\sqrt{A_3(I)}r)}{A_3(I)^{1/2}}\right],
\nonumber
\end{eqnarray}
where we have defined $A_3(I)=-4K_1(I  )\geq 6$. 
This proves~\eqref{1.4a}.

In order to see that inequality~\eqref{1.4b} holds,
just observe that by item~\ref{it8} of Theorem~\ref{mainStructure},
\begin{equation}
\mbox{\rm Area}(\cup_{i=1}^k\Delta_i)\leq  2\pi \sum_{i=1}^k m(i)r^2_F(i)
\leq  2\pi S\frac{\de^2}{4}\leq \frac{\pi S}{8}\leq \frac{3\pi}{8}I,
\end{equation}
hence
\begin{eqnarray}
\hspace{-1.1cm}\mbox{\rm Area}\left[ B_M(x,r)\right]
&\leq&
\mbox{\rm Area}\left[ B_M(x,r)\setminus (\cup_{i=1}^k\Delta_i)\right]
+\mbox{\rm Area}(\cup_{i=1}^k\Delta_i)
\nonumber
\\
&\leq& \frac{2(6\pi+1)}{A_3(I)}I\left[ 2[\cosh(\sqrt{A_3(I)}r)-1]
+\frac{\sinh(\sqrt{A_3(I)}r)}{A_3(I)^{1/2}}\right]
+\frac{3\pi}{8}I
\nonumber
\\
&=&
I\left[ \frac{2(6\pi+1)}{A_3(I)}\left( 2[\cosh(\sqrt{A_3(I)}r)-1]
+\frac{\sinh(\sqrt{A_3(I)}r)}{A_3(I)^{1/2}}\right)
+\frac{3\pi}{8}\right], \label{3.29}
\end{eqnarray}
from where one deduces~\eqref{1.4b}.

To finish this section, we prove~\eqref{1.5'}. Let us denote by
$h(r)$ the RHS of~\eqref{3.29}. Then, taking $r=D:=
\mbox{Diameter}(M)$ we have $B_M(x,D)=M$ for any $x\in M$, and so,
\begin{equation}\label{3.30}
C_1(I)(g+1)\stackrel{\eqref{1.1}}{\leq }
\mbox{Area}(M)=\mbox{Area}\left[ B_M(x,D)\right]\leq h(D).
\end{equation}
Since $\sinh(t)\leq \cosh(t)$,
\begin{eqnarray}
h(D)&\leq &I\left[ \frac{2(6\pi+1)}{A_3(I)}\left( 2[\cosh(\sqrt{A_3(I)}D)-1]
+\frac{\cosh(\sqrt{A_3(I)}D)}{A_3(I)^{1/2}}\right)
+\frac{3\pi}{8}\right]
\nonumber
\\
&\stackrel{(A_3(I)\geq 6)}{\leq} & I\left[ \frac{6\pi+1}{3}\left( 2[\cosh(\sqrt{A_3(I)}D)-1]
+\frac{\cosh(\sqrt{A_3(I)}D)}{\sqrt{6}}\right)
+\frac{3\pi}{8}\right] \nonumber
\\
&= & I\left[ \frac{6\pi+1}{3}\left( \left(2+\frac{1}{\sqrt{6}}\right)\cosh(\sqrt{A_3(I)}D)-2
\right)
+\frac{3\pi}{8}\right] \nonumber
\\
&< & I\frac{6\pi+1}{3}\left(2+\frac{1}{\sqrt{6}}\right)\cosh(\sqrt{A_3(I)}D).
\label{3.31}
\end{eqnarray}
\eqref{3.30} and \eqref{3.31} give
\[
C_1(I)(g+1)\leq I\frac{6\pi+1}{3}\left(2+\frac{1}{\sqrt{6}}\right)\cosh(\sqrt{A_3(I)}D),
\]
from where inequality~\eqref{1.5'} follows directly.
This completes the proof of Theorem~\ref{main22}.

\subsection{Proof of item~\ref{It1b} of the Theorem~\ref{main2}}
Consider an element $(F\colon M\la X)\in \L(I, 1,1,1,1)$.
As in previous sections, we may assume that $M$ is compact and
connected, and let $g=g(M)$ be the genus of $M$.

\begin{claim}\label{claim3.10}
If $I=0$, then~\eqref{1.1b} holds with $G(0)=0$.
\end{claim}
\begin{proof}
Since $M$ is stable, we have $|A_M|\leq C_s$ in $M$
by~\eqref{eqcurvestim2}. Thus, the Gauss equation implies that
the Gaussian curvature $K$ of $M$ satisfies $K\geq -1-\frac{1}{2}C_s^2$.
Using the Gauss-Bonnet theorem,
\[
\textstyle{(1+\frac{1}{2}C_s^2)}\mbox{Area}(M)\geq -\int_MK=-2\pi \chi(M)\geq 2\pi (g-1).
\]
Hence,
\[
\mbox{Area}(M)\geq \frac{2\pi}{1+\frac{1}{2}C_s^2}(g-1),
\]
which is strictly bigger than $\frac{\pi}{3+4C_s+4C_2^2}(g+1)$ when
$g\geq 2$. Consequently, \eqref{1.1b} holds whenever $g\geq 2$.
To finish the proof of the claim it remains to check
that~\eqref{1.1b} holds for $g=0,1$, which we do next.
\[
\mbox{Area}(M)\stackrel{\mbox{\footnotesize (item~\ref{It0})}}{\geq }
C_A\stackrel{\mbox{\footnotesize (a)}}{>}\frac{2\pi}{3+4C_s+4C_2^2}\stackrel{\mbox{\footnotesize (b)}}{\geq }\frac{2\pi}{3+4C_s+4C_2^2}(g+1)
\]
where in (a) we have used that $C_s\geq 2\pi$, and in (b) that $g\leq 1$.
Now the claim is proved.
\end{proof}

By Claim~\ref{claim3.10}, it remains to prove item~\ref{It1b} of 
Theorem~\ref{main2} assuming $I\geq 1$.
The additional assumption $g\geq 12I-3$ guarantees,
by \eqref{3.7}, that
\begin{equation}
	\int _{\wt{M}}K	=\int _{{M}}K -\int _{\cup_{i=1}^k\Delta_i}K\leq -\pi (g+1).
	\label{3.7a}
\end{equation}

By Lemma~7.1 in~\cite{mpe18},
there exists a positive constant $\wh{C}_s(1)$, which in our setting is
$1+2C_s$, such that if $\sup |A_M|
>\wh{C}_s(1)$, then there exists a nonempty finite subset
$\{ q_1,\ldots ,q_n\} \subset  M$
with $1\leq n\leq I$, such that
\begin{enumerate}
	\item $|A_M|$ achieves its maximum in $M$ at $q_1$, and for
	$i=2,\ldots ,n$, $|A_M|$ achieves its maximum in
	$M\setminus [B_M(q_1,1)\cup \ldots \cup
	B_M(q_{i-1},1)]$ at $q_i$.
	\item For each $i=1,\ldots ,n$, $|A_M|(q_i)>\wh{C}_s(1)$
	and  the  intrinsic
	balls $B_M(q_i,1/2)$ are pairwise disjoint and unstable.
	\item $|A_M|\leq \wh{C}_s(1)$ in $M\setminus
	[B_M(q_1,1)\cup \ldots \cup B_M(q_{n},1)]$.
\end{enumerate}
We next define a partition of the surface $\wt{M}$ that appears
in item~\eqref{it7} of Theorem~\ref{mainStructure}.
\begin{itemize}
\item If $\sup |A_{M}|\leq \wh{C}_s(1)$, let $\wt{M}_1=\varnothing$ and $\wt{M}_2=\wt{M}$.
\item Otherwise, let
$\wt{M}_1=\wt{M} \cap [\cup_{i=1}^n B_M(q_i,1)]$ and
$\wt{M}_2=\wt{M}\setminus [\cup_{i=1}^n B_M(q_i,1)]$.
\end{itemize}
In particular, the second fundamental form of the surface $\wt{M}_2$
satisfies $|A_{\wt{M}_2}|\leq \wh{C}_s(1)$.

By the discussion around inequality~\eqref{3.2},
the Gaussian curvature function
of $\wt{M}_1$ satisfies $K_{\wt{M}_1}\geq K_1$
(where $K_1=K_1(I)\leq -3/2$ is defined in~\eqref{3.3a}),
and the  Gaussian curvature function of $\wt{M}_2$
satisfies $K_{\wt{M}_2}\geq -1-\frac{1}{2}\wh{C}_s(1)^2$.
Also, by inequalities~\eqref{1.2a} and~\eqref{1.4b},
there exists an explicit function $h\colon \N \to (0,\infty)$
such that
\begin{equation}\label{3.37}
\mbox{Area}[\wt{M}\cap [\cup_{i=1}^n B_M(q_i,1)]
\leq \sum_{i=1}^n\mbox{Area}B_M(q_i,1)\leq h(I).
\end{equation}

Therefore,
\begin{equation}
{\textstyle
K_1(I)h(I)- \left[1+\frac{1}{2}\wh{C}_s(1)^2\right]
\mbox{\rm Area}(\wt{M}_2)}
\stackrel{\eqref{3.37}}{\leq}
	\int_{\wt{M}_1} K_{\wt{M}_1} +\int_{\wt{M}_2}K_{\wt{M}_2}=\int _{\wt{M}}K	\stackrel{\eqref{3.7a}}{\leq} -\pi (g+1).
\end{equation}
Solving for the area of $\wt{M}_2$, we have
\begin{equation}
	\label{3.39}
\mbox{\rm Area}(\wt{M}_2)\geq \frac{\pi(g+1)+K_1(I)h(I)}{1+\frac{1}{2}\wh{C}_s(1)^2}.
\end{equation}
After setting  $C:=\frac{\pi/2}{1+\frac{1}{2}\wh{C}_s(1)^2}=\frac{\pi}{3+4C_s+4C_s^2}$,
we get the estimate
\begin{equation}\label{3.15c}
\mbox{\rm Area}(M)\geq\mbox{\rm Area}(\wt{M})\geq
\mbox{\rm Area}(\wt{M}_2)	\geq
C(g+1)+\frac{\frac{\pi}{2}(g+1)+K_1(I)h(I)}{1+\frac{1}{2}\wh{C}_s(1)^2}.
\end{equation}
Define
\[
G(I):=\max\left\{ 12I-3,\left\lceil \frac{-2K_1(I)h(I)}{\pi}\right\rceil
-1\right\}\in  \N,
\]
where for a real number $x$, we denote by $\lceil x\rceil$ the
smallest integer that is not smaller than $x$ (also known as the
{\it ceiling function} at $x$). Then, whenever $g(M)\geq  G(I)$ we have
that the second term in the RHS of~\eqref{3.15c} is non-negative,
which completes the proof of item~\ref{It1b} of Theorem~\ref{main2}.

\subsection{Proof of item~\ref{It3b} of Theorem~\ref{main2}}
Recall that in Section~\ref{sec3.1} we normalized the space
$\Lambda$, passing from an $H$-immersion $(F\colon M\la X)\in \Lambda
(I,H_0,r_0,1,K_0)$ to the immersion
$(F'\colon M\la X')\in \L(I,1,1,1,1)$,
where $\l$ is given by~\eqref{3.0} and $X'$ is the Riemannian manifold
obtained after scaling the original metric of $X$ by $\sqrt{\l}$.
Observe that $F'$ has mean curvature $H'=H/\l$ and $X'$ has
scalar curvature $\rho'=\rho/\l^2$.

Suppose that the scalar curvature $\rho$ of
$X$ satisfies $3H^2+\frac{1}{2}\rho \geq c$	in $X$ for some $c>0$, where
$H$ is the mean curvature of an immersion $(F\colon M\la X)\in
\L(I,H_0,r_0,1,K_0)$.
In this setting, Rosenberg~\cite{rose4} (see also~\cite[Theorem~2.12]{mpr19}) proved that
every stable subdomain $\Omega \subset M$ satisfies
\begin{equation}
\label{3.17}
d_M(x,\partial \Omega)\leq \frac{2\pi}{\sqrt{3c}}:=R_c,
\end{equation}
for all $x\in \Omega$. Since $3(H')^2+\frac{1}{2}\rho'\geq c':=c/\l^2$,
the estimate~\eqref{3.17} applied to the same stable
subdomain $\Omega$ viewed inside the domain of $F'$ gives
that the intrinsic distance in the metric induced by $F'$
from any $x\in \Omega$ to $\partial \Omega$ is at most
$\frac{2\pi}{\sqrt{3c'}}=\frac{2\pi \l}{\sqrt{3c}}$. This linear
scaling on the upper bound for the intrinsic radius of stable subdomains
allows us to reduce item~\ref{It3b} of Theorem~\ref{main2} to the
following statement.
\begin{proposition}
	\label{propos3.7}
Let $F\colon M\la X$ be an $H$-immersion in $\L(I, 1,1,1,1)$,
where $M$ is connected.
If the scalar curvature $\rho$ of $X$ satisfies $3H^2+\frac{1}{2}\rho
\geq c$	in $X$ for some $c>0$, then $M$ is compact, and
there exists $A_2(I,c)>0$ such that
\begin{equation}
\label{1.7'}
\mbox{\rm Area}(M)\leq A_2(I,c)  \qquad
\mbox{\rm Diameter}(M)\leq 2(I+1)R_c, \qquad
g(M)\leq \frac{A_2(I,c)}{C_1(I)}-1.
\end{equation}
\end{proposition}
\begin{proof}
We first show that $M$ is compact. Arguing by contradiction,
suppose $M$ is non-compact. Since $M$ is complete,
there is a geodesic ray in $M$, i.e., an embedded,
length-minimizing unit-speed geodesic arc $\g\colon[0,\infty)\to M$. Consider the infinite collection
\[
\cC(n)=\{ B_M(\g(2jn),n)\ | \ j\in \N\cup \{ 0\}\}
\]
of pairwise disjoint open intrinsic balls in $M$.
Since the index of $M$ is at most $I$, then the subcollection
of unstable balls in $\cC(n)$ is finite. This
implies that $M$ contains stable balls of arbitrarily large
radius, a property which contradicts that for $r>R_c$,
$B_M(x,r)$ cannot be stable as follows from~\eqref{3.17}.
Therefore, $M$ is compact.

We will divide the proof of~\eqref{1.7'} into two claims.
\begin{claim}
\label{claim3.8}
{\rm Diameter}$(M)\leq 2(I+1)R_c$ (i.e., the second inequality
in~\eqref{1.7'} holds).
\end{claim}
\begin{proof}[Proof of Claim~\ref{claim3.8}]
Arguing by contradiction, suppose that there exist points $p,q\in M$
at intrinsic distance $L:=d_M(p,q)>2(I+1)R_c$. Let $\G\colon [0,L]
\to M$ be a geodesic arc parameterized by arc length, such that $\G(0)=p$
and $\G(L)=q$. Choose $R>R_c$ such that $2(I+1)R\leq L$.
Consider the following collection of $I+1$
pairwise disjoint open intrinsic balls in $M$, see Figure~\ref{fig3}.
\[
\cC'(n)=\{ B_M\left(\G((2j-1)R),R\right)\ | \ j=1,\ldots ,I+1\}.
\]
\begin{figure}
\begin{center}[h]
\includegraphics[width=15cm]{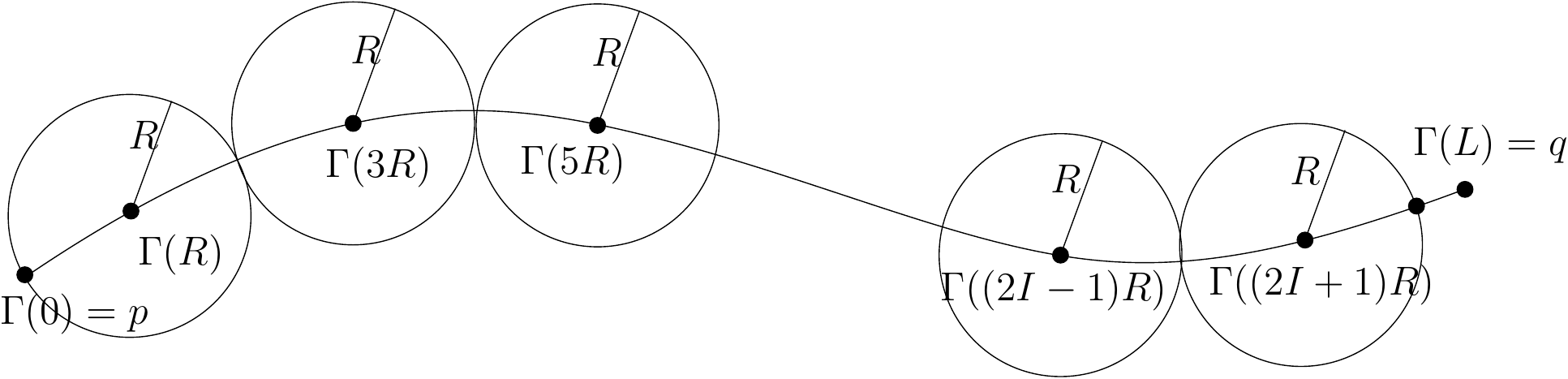}
\caption{The pairwise disjoint collection of metric balls in $\cC'(n)$.
Observe that the boundary of the last ball in the chain,
$\partial B_M(\G((2I+1)R,R)$, intersects the image of $\G$
at the points $\G(2IR)$, $\G(2(I+1)R)$, and that $2(I+1)R\leq L$ by construction.}
\label{fig3}
\end{center}
\end{figure}

Since the index of $M$ is at most $I$ and the $I+1$ balls in $\cC'(n)$
are pairwise disjoint, we deduce that at least one of these balls is
stable. This contradicts that $R>R_c$ and~\eqref{3.17}. This
contradiction proves the claim.
\end{proof}
\begin{claim}
\label{claim3.9}
Let $\wt{h}=\wt{h}(I,r)
\colon (\N\cup \{ 0\})\times (0,\infty) \to (0,\infty)$
be the maximum of the right-hand-sides of~\eqref{1.2a} and 
\eqref{1.4b}. Then, the first and third inequalities in~\eqref{1.7'}
hold for $A_2(I,c)=\wt{h}(I,2(I+1)R_c)$.
\end{claim}
\begin{proof}[Proof of Claim~\ref{claim3.9}]
Observe that $r\mapsto \wt{h}(I,r)$ is increasing. Take $x\in M$. By
the already proven inequalities~\eqref{1.2a}
and~\eqref{1.4b}, we have $\mbox{Area}(B_M(x,r))\leq \wt{h}(I,r)$
for all $r>0$. Applying this estimate to the choice
$r=D:=\mbox{Diameter}(M)<\infty$ (observe that $M=B_M(x,D)$) and
using that $\wt{h}(I,r)$ is increasing in $r$, we get
\begin{equation}
\label{3.20}
\mbox{\rm Area}(M)\leq \wt{h}(I,D)\stackrel{\mbox{\footnotesize (Claim~\ref{claim3.8})}}{\leq }
\wt{h}(I,2(I+1)R_c)=A_2(I,c).
\end{equation}
and thus, the first inequality in~\eqref{1.7'} holds.
As for the third one,
it is clearly equivalent to proving that $C_1(I)(g(M)+1)\leq A_2(I,c)$.
Applying~\eqref{1.1} we have
\[
C_1(I)(g(M)+1)\leq \mbox{Area}(M)
\stackrel{\eqref{3.20}}{\leq }A_2(I,c),
\]
and the proof of the claim is complete.
\end{proof}

%

Claims~\ref{claim3.8} and~\ref{claim3.9} prove~\eqref{1.7'},
which finishes the proof of Proposition~\ref{propos3.7}, and
consequently item~\ref{It3b} of Theorem~\ref{main2} is also proved.
\end{proof}
\begin{remark}
{\rm
Since the function $\wt{h}=\wt{h}(I,r)$ appearing in the proof of
Claim~\ref{claim3.9} is increasing in $r$, and $R_c$ is decreasing in $c$,
we deduce that $A_2(I,c)=\wt{h}(I,2(I+1)R_c)$ is decreasing in $c$.
This indicates that if we relax the hypothesis $3H^2+\frac{1}{2}\rho
\geq c$	in Proposition~\ref{propos3.7} by taking $c\to 0^+$, then
the estimates for the area, diameter and genus of $M$ in
Proposition~\ref{propos3.7} get worse (in fact, $\lim_{c\to 0^+}
A_2(I,c)=\lim_{c\to 0^+}
R_c=\infty$).
}
\end{remark}

\center{William H. Meeks, III at  profmeeks@gmail.com\\
		Mathematics Department, University of Massachusetts, Amherst, MA 01003}
\center{Joaqu\'\i n P\'{e}rez at jperez@ugr.es\\
		Department of Geometry and Topology and Institute of Mathematics
		(IMAG), University of Granada, 18071, Granada, Spain}
\bibliographystyle{plain}
\bibliography{bill}
\end{document}